\documentclass[onefignum,onetabnum]{siamart190516}
\pdfoutput=1

\usepackage[utf8]{inputenc}
\usepackage{amsfonts,amsmath}
\usepackage{comment}
\usepackage{array,booktabs}
\usepackage{graphicx,epstopdf}
\usepackage{pdflscape}
\usepackage[caption=false]{subfig}

\usepackage{pstricks}
\usepackage{float}
\usepackage{hyperref}
\usepackage{microtype}
\usepackage{pdfpages}

\usepackage{color}

\newcommand{\bone}{{\bf 1}}
\newcommand{\evec}{{\bf e}}

\newcommand{\xvec}{{\bf x}}

\newcommand{\svec}{{\bf s}}

\newcommand{\tvec}{{\bf t}}

\newsiamremark{remark}{Remark}

\title{Mittag--Leffler functions and their applications in network science\thanks{Submitted to the editors DATE.
\funding{The work of F.A. was supported by fellowship ECF-2018-453 from the Leverhulme Trust. The work of F.D. was supported by the INdAM GNCS 2020 Project ``\emph{Nonlocal models for the analysis of complex networks}''.
}}}

\author{Francesca Arrigo\thanks{Department of Mathematics and Statistics, University of Strathclyde, Glasgow, UK (\email{francesca.arrigo@strath.ac.uk},).}
\and Fabio Durastante\thanks{Dipartimento di Matematica, Universit\`a di Pisa, Pisa, IT (\email{fabio.durastante@di.unipi.it}), Istituto per le Applicazioni del Calcolo ``M. Picone'', Consiglio Nazionale delle Ricerche, Napoli, IT (\email{f.durastante@na.iac.cnr.it}).}}

\headers{Mittag--Leffler functions in Network Science}{F. Arrigo and F. Durastante}

\begin{document}

\maketitle

\begin{abstract}
We describe a complete theory for walk-based centrality indices in complex networks defined in terms of Mittag--Leffler functions. This overarching theory includes as special cases well-known centrality measures like subgraph centrality and Katz centrality. 
The indices we introduce are parametrized by two numbers; by letting these vary, we show that Mittag--Leffler centralities interpolate between degree and eigenvector centrality, as well as between resolvent-based and exponential-based indices.
We further discuss modelling and computational issues, and provide guidelines on parameter selection. The theory is then extended to the case of networks that evolve over time. Numerical experiments on synthetic and real-world networks are provided.  
\end{abstract}



	\section{Introduction}
Networks (or graphs) have become an increasingly popular modelling tool in a range of applications, often where the question of interest to practitioners is to identify the most important entities (which can be nodes, edges, sets of nodes, etc.) within the system under study; see, e.g.,~\cite{Ashtiani2018,Holme2003,newman2018networks,Vargas2018,VOLTESDORTA2017119}.  This question is commonly answered by means of centrality measures; These are functions that assign nonnegative scores to the entities, with the understanding that the higher the score, the more important the entity.   
Several centrality measures have been introduced over the years~\cite{boldi2014axioms,bonacich1987power,MR2171832,katz1953new}. 
Here we consider walk-based centrality indices~\cite{MR2736969}, where a walk around a graph is a sequence of nodes that can be visited in succession following the edges in the graph. These measures can be defined using (sums of) entries of matrix functions described in terms of the adjacency matrix $A$ of the graph and assign scores to nodes based on how well they spread information to the other nodes in the network. Possibly the most widely known measures of centrality in this family are Katz centrality~\cite{katz1953new}, defined for node $i$ as the $i$th entry of $(I-\gamma A)^{-1}\bone$, for $0<\gamma\rho(A)<1$ and $\bone$ the vector of all ones, and subgraph centrality~\cite{MR2171832}, defined for node $i$ as $(e^{\gamma A})_{ii}$, for $\gamma>0$. 
The popularity of these measures stems from their interpretability in terms of walks around the graph, but it also follows from the fact that they are easily computed or approximated; see, e.g.,~\cite{MR3143833,HigFM}. 
Another interesting feature of these measures was shown in \cite{MR3354998}, where the authors proved that a special class of functions, which includes the exponential and the resolvent, induces centrality indices  that interpolate between degree centrality, defined as the number of connections that a node has, and eigenvector centrality, defined using the entries of the Perron eigenvector of $A$. 

In the following we show that Mittag--Leffler (ML) functions~\cite{mittag1903nouvelle}, which fall in the aforementioned class of functions, induce well-defined centrality measures that moreover interpolate between resolvent-based and exponential-based indices, thus closing the gap between the two induced centralities.  
Several instances of ML centrality indices are scattered throughout the network science literature, but often they are not being identified as such. One of the contributions of this work is to provide an exhaustive (to the best of our knowledge) review of such appearances. 
Furthermore, 
this work  provides a thorough  analysis of the properties of parametric ML centrality indices and a characterization of the possible choices of parameters that ensure both interpretability and computability of such measures. 
The results are then extended to the case of temporal network, following the contents of~\cite{MR3177230}.

\smallskip

Our contribution is thus threefold. We provide an extensive review of  previous appearances of Mittag--Leffler centrality indices in network science; We develop a general theory for such measures and further show that they ``close the gap" between resolvent-based centrality measures and exponential-based centrality measure, and we provide guidelines for parameter selection; Finally, we describe extensions of such centrality measures to networks that evolve over time. 

The paper is organized as follows. 
In \cref{sec:background} we review some basic definitions and tools from graph theory that will be used throughout. We also review the definition of ML functions and provide some examples of functions in this family. 
In \cref{sec:MLindices} we review previous appearances of ML centrality and communicability indices, discuss interpretability issues, and introduce the new centrality indices. 
We further perform numerical tests on some real-world networks. \Cref{sec:temporal} describes how ML centrality indices can be adapted to the case of time-evolving networks, extending results from~\cite{MR3177230} to a more general framework. Numerical results on synthetic and real-world networks are also discussed.
We conclude with some remarks and a brief description of future work in \cref{sec:conclusions}

\section{Background}\label{sec:background}
	
	This section is devoted to a brief introduction of the main concepts that will be used throughout the paper. In particular, we review basic concepts from graph theory and network science; we also recall the definition of Mittag--Leffler functions and a few of their properties.  
	
	\subsection{Graphs}
	
	A {\it graph} or {\it network} $G = (V,E)$ is defined as a pair of sets:  a set $V = \{1,2,\ldots,n\}$ of {\it nodes} or {\it vertices} and  a set $E\subset V\times V$ of {\it edges} or {\it links} between them~\cite{bapat2010graphs}. 
	If the set $E$ is symmetric, namely if for all $(i,j)\in E$ then $(j,i)\in E$, the graph is said to be \textit{undirected}; {\it directed} otherwise. 
	An edge from a node to itself is called a {\it loop}. 
	
	A popular way of representing a network is via its {\it adjacency matrix} $A = (a_{ij})\in\mathbb{R}^{n\times n}$, entrywise defined as
	\[
	a_{ij} = 
	\begin{cases}
	w_{ij} & \text{if }(i,j)\in E \\
	0 & \text{otherwise}
	\end{cases}
	\]
    where $w_{ij}>0$ is the weight of edge $(i,j)$. 
    In this paper we will restrict our attention to unweighted \textit{simple} graphs, i.e., graphs that are undirected and do not contain loops or repeated edges between nodes, and for which the weights of the edges are all uniform; consequently, the adjacency matrices used throughout this paper will be binary, symmetric, and with zeros on the main diagonal. 
    We note however that all the results in this paper can be generalized beyond this simple case.

	\subsection{Centrality measures}
	One of the most addressed questions in network science concerns the identification of the most important entities within the graph; What is the most vulnerable airport to a terror attack~\cite{VOLTESDORTA2017119}? Which is the road more likely to be busy during rush hour~\cite{Holme2003}? Who is the most influential pupil in the school~\cite{Vargas2018}? What proteins are vital to a cell~\cite{Ashtiani2018}? 
	Several strategies to answer these questions have been presented over the years, and these all rely on the idea that an entity is more important within the graph if it is better connected than the others to the rest of the network. In order to quantify this idea of importance, entities are assigned a nonnegative score, or {\it centrality}~\cite{boldi2014axioms}: the higher its value, the more important the entity is within the graph. 
    We will focus here on centrality measures for nodes, although we note that several centrality measures for edges have been defined over the years~\cite{MR3479697,MR3914214} and that everything discussed here for nodes easily translates to address the case of edges by working on the line graph~\cite{bapat2010graphs}.
The simplest measure of centrality for nodes is {\it degree centrality}. According to this measure, a node $i$ is more important the larger the number of its connections $d_i= \sum_{j=1}^n a_{ij} = (A\bone)_i$, where $\bone$ is the vector of all ones. 
    This measure is very local, in the sense that it is oblivious to the whole topology of the network and thus may misrepresent the role of nodes: a node acting as the only bridge between two tightly connected sets of nodes has low degree, but it has extremely high importance as its failure would cause the network to break into two separate components. 
A way around this issue is to consider both the number of neighbors and their importance when assigning scores to nodes; see, e.g., \cite{vigna2016spectral} and references therein.  
    The centrality measure formalizing this idea is known as {\it eigenvector centrality}~\cite{bonacich1972factoring,bonacich1987power}; it is  entrywise defined as:
    \[
    x_i = \frac{1}{\rho(A)}\sum_{j=1}^n a_{ij} x_j
    \]
    where $\rho(A)>0$ is the spectral radius of the irreducible adjacency matrix $A\geq 0$. Existence, uniqueness and nonnegativity of the vector $\xvec = (x_i)$ are guaranteed by the Perron-Frobenius theorem; see, e.g.,~\cite{horn2012matrix}.

    Degree and eigenvector centrality represent the two limiting behaviors of a wider class of parametric centrality measures that can be defined in terms of matrix functions~\cite{MR2736969}.\footnote{This result was shown in a paper by Benzi and Klymko~\cite{MR3354998} and later extended to the non-backtracking framework in~\cite{arrigo2020beyond}.} Consider the analytic function $f$ defined via the following power series:  
	\[f(z) = \sum_{r=0}^\infty c_r z^r\] 
	{with} $c_r\geq 0$ and $|z|<R_f${, where $R_f$ the radius of convergence of the series, which can be either finite or infinite}; then under suitable hypothesis on the spectrum of $A$~\cite[Theorem 4.7]{HigFM}, we can write: 
	\[f(A) = \sum_{r=0}^\infty c_r A^r.\]
Recall that a {\it walk} of length $r$ is a sequence of $r+1$ nodes $i_1, i_2,\ldots, i_{r+1}$ such that $(i_\ell,i_{\ell+1})\in E$ for all $\ell = 1,\ldots,r$; moreover, it is easy to show the number of such walks is $(A^r)_{i_1,i_{r+1}}$~\cite{bapat2010graphs}.  Therefore, entrywise, this matrix function has a clear interpretation in terms of walks taking place across the graph: $(f(A))_{ij}$ is a weighted sum of the number of all walks of any length that start from node $i$ and end at node $j$. Since the weights are such that $c_r\to 0$ as $r$ increases, we are also tacitly assuming that walks of longer lengths are considered to be less important. 
	In~\cite{MR2171832} the authors defined the {\it subgraph centrality} of a node $i\in V$ as 
	\begin{equation*}\label{eq:SC_f}
	s_i(f) = \evec_i^T f(A)\evec_i = \sum_{r=0}^\infty c_r (A^r)_{ii}.
	\end{equation*}
	This measure accounts for the returnability of information from a node to itself and it is a weighted count of all the subgraphs node $i$ is involved in; see, e.g., \cite{EstradaBook}. 
	We will write $\svec(f) = (s_i(f))$ to denote the vector of subgraph centralities induced by the function $f$.

    The most popular functions used in networks science are $f(z) = e^z$ \cite{MR2171832} and $f(z) = (1+z)^{-1}$~\cite{katz1953new}; however nothing in principle forbids the use of other analytic functions~\cite{MR3479697,MR3948252,MR3005305}. 
    
    Subgraph centrality is computationally quite expensive to derive for all nodes, since one has to compute all the diagonal entries of $f(A)$ and this is usually unfeasible for large networks. However, if only a few top ranked nodes need to be identified, approximation techniques are available; see, e.g.,\cite{MR3143833}.
    
    In~\cite{benzi2013total} the authors introduced the concept of {\it total (node) communicability}. Here, the importance of a node depends on how well it communicates with the whole network, itself included:     
    \[
        \tvec(f) = f(A)\bone.
    \]
   Entrywise it is thus defined as
   \begin{equation*}\label{eq:TC_f}
        t_i(f) = \sum_{j=1}^n (f(A))_{ij} = \sum_{j=1}^n\sum_{r=0}^\infty c_r(A^r)_{ij}
    \end{equation*}
    
    Computationally speaking, this measure can be computed more efficiently than subgraph centrality, and can also be easily updated after the application of low-rank modification of the adjacency matrix $A$, i.e., after the removal or the addition of few edges~\cite{MR3780747,MR3863072}.

    \begin{remark}
    All the above definition have been given in the setting of unweighted networks where the weight assigned to the edges is assumed to be unitary. If $A$ is replaced in the above definition by $\gamma A$, for some appropriate $\gamma\in (0,1)$, the definitions continue to make sense and we are then working with {\it parametric} versions of subgraph centrality and total communicability. 
    \end{remark}

	In the next section we recall the definition of the Mittag--Leffler function and a few properties that will be used in this paper. 
	
\subsection{Mittag--Leffler Functions}\label{sec:themlfunction}
	The family of {\it Mittag--Leffler  (ML) functions}  is a family of analytic functions $E_{\alpha,\beta}(z)$ that were originally introduced in~\cite{mittag1903nouvelle}. For each choice of $\alpha,\beta>0$ they are defined as follows
	\begin{equation} \label{eq:MLfunction}
		E_{\alpha,\beta}(z) = \sum_{r=0}^\infty c_r(\alpha,\beta) z^r= \sum_{r=0}^{\infty} \frac{z^r}{\Gamma(\alpha r + \beta)},
	\end{equation}
	where $c_r(\alpha,\beta)=\Gamma(\alpha r + \beta)^{-1}$ and $\Gamma(z)$ is the {\it Euler Gamma function}:
	$$
	\Gamma(z) = \int_{0}^{\infty}t^{z-1}e^{-t}dt.
	$$

\begin{table}[t]
\centering
		\footnotesize
					\caption{Closed form expression of the Mittag--Leffler function $E_{\alpha,\beta}(z)$ for selected values of $\alpha$ and $\beta$.}\label{tab:closed_forms}
		\begin{tabular}{ccll}
			\toprule
			$\alpha$ & $\beta$ & Function & \\
			\midrule
			0 & 1 & $(1-z)^{-1}$ & Resolvent \\
			1 & 1 & $\exp(z)$ & Exponential \\
			1/2 & 1 & $\exp(z^2)\operatorname{erfc}(-z)$ & Error Function\footnotemark \\
			2 & 1 & $\cosh(\sqrt{z})$ & Hyperbolic Cosine \\
			2 & 2 & $\sinh(\sqrt{z})/\sqrt{z}$ & Hyperbolic Sine \\
			4 & 1 & $1/2 [\cos(z^{1/4}) + \cosh(z^{1/4})]$ & \\
		1 & $k=2,3,\ldots$ &$z^{1-k}(e^z - \sum_{r=0}^{k-2} \frac{z^r}{r!})$& $\varphi_{k-1}(z) = \sum_{r=0}^\infty \frac{z^r}{(r+k-1)!}$ \\
			\bottomrule
		\end{tabular}
	\end{table}
		\footnotetext{$\operatorname{erfc}(z) = 1 - \operatorname{erf}(z)$ is complementary to the error function $\operatorname{erf}(z) = \frac{2}{\sqrt{\pi}}\int_0^z e^{-t^2}dt$.}

For particular choices of $\alpha,\beta>0$, the ML function $E_{\alpha,\beta}(z)$ have nice closed form descriptions. For example, when $\alpha=\beta=1$ we have that $E_{1,1}(z) = {\rm exp}(z)$, since  $\Gamma(1) = \Gamma(2) = 1$ and $\Gamma(r+1) = r\,\Gamma(r) = r!$ for all $r\in\mathbb{N}$. 
We list a few of  these closed form expressions for ML functions in \cref{tab:closed_forms}.

Our goal is to use this family of functions to define new walk-based centrality indices. 
We will focus on the case when $\beta = 1$ and we will adopt from now on the notation $E_{\alpha}(z) = E_{\alpha,1}(z)$. 

Before proceeding, we make two remarks. Firstly, $\Gamma(\alpha r + \beta)>0$ for every $\alpha\geq 0$,  $\beta > 0$, and $r \geq 0$; secondly, the function $g(r) := \Gamma(\alpha r + \beta)$ is not monotonic. In \cref{fig:coeff_gamma} we plot the values of $\Gamma(\alpha r+1)^{-1}$ for $r=0,1,2,3$ and $\alpha = 0,0.1,\ldots,1$. 
	\begin{figure}[t]
	    \centering
	    \includegraphics[width = 0.5\textwidth]{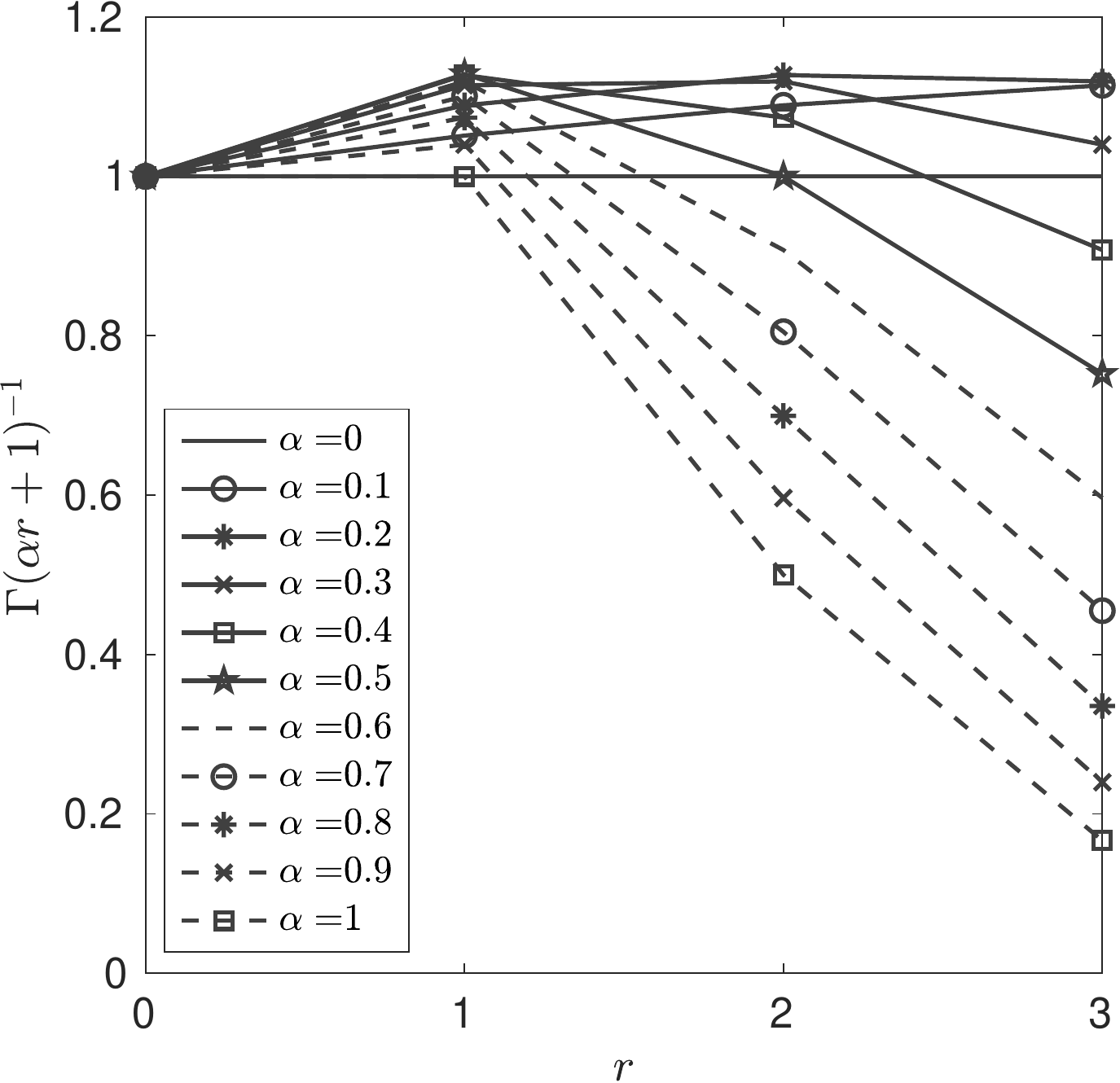}
	    \caption{Plot of $\Gamma(\alpha r+1)^{-1}$ for $r=0,1,2,3$ and $\alpha = 0,0.1,\ldots,1$.}
	    \label{fig:coeff_gamma}
	\end{figure}
	Non-monotonicity of coefficients is not a problem \textit{per se}, however we note that it is customary in network science to define walk-based centrality measures that employ analytic functions with monotonically decreasing coefficients. The reason for this is to foster the intuition that shorter walks should be given more importance than longer ones, because they allow for information to travel faster (i.e., by taking fewer steps) from the source to the target. 
	The fact that the coefficients in the power series expansion of $E_\alpha(z)$ for $\alpha\geq 0$ are not monotonic is something that we will need to be aware of when defining centrality indices for entities in networks. In \cref{lem:monotoniccoeffs} below we will describe how to suitably select a scaling of the adjacency matrix $A$ to ensure monotonicity of the coefficients.
 	
\section{Mittag--Leffler based network indices}\label{sec:MLindices}
    We want to ``close the gap" between resolvent based centrality measures, defined in terms of $f(z) = (1-z)^{-1} = E_{0}(z)$ and  exponential based centrality measures, defined in terms of $f(z) = e^z = E_{1}(z)$. 
    The former function has a discontinuity at $ z= 1$, whilst the latter is entire; however they both can be represented as ML functions. 
    In the following we will 
    \begin{itemize}
        \item review previous appearances of ML functions in network science; 
        \item show that it is possible to describe centrality measures in terms of entries (or sum of entries) of $E_{\alpha}(\gamma A)$ for values of $\alpha$ other than $0$ and $1$ and for suitably selected $\gamma>0$; and
        \item show numerically that careful selection of the parameters $\alpha$ and $\gamma$ allows ML functions to detect information not encoded by degree or eigenvector centrality.
    \end{itemize}
    
    \subsection{Previous appearances of Mittag--Leffler functions}
    We begin by noticing that ML functions have already been employed in the network science literature, often without being recognized as such. The most renowned instances are the previously mentioned exponential and resolvent based centrality measures, introduced in~\cite{MR2171832} and~\cite{katz1953new}, respectively. 
    However, other ML functions have been used. 
    In \cite{MR3005305} the authors introduce new centrality and communicability indices for directed networks by exploiting the representation of such networks as bipartite graphs; see \cite{brualdi1980bigraphs}. 
    In particular, the authors recast the discussion of walk-based centrality measures for directed graph with adjacency matrix $A$ in terms of the symmetric block matrix 
\[
\cal{A} = \begin{bmatrix} 0 & A \\ A^T & 0\end{bmatrix}.
\]    
After showing that 
\[
e^{\cal{A}} = 
\begin{bmatrix}
\cosh(\sqrt{AA^T}) & A(\sqrt{A^TA})^\dagger\sinh(\sqrt{A^TA}) \\
\sinh(\sqrt{A^TA})(\sqrt{A^TA})^\dagger A^T & \cosh(\sqrt{A^TA})) 
\end{bmatrix}
\]
where the superscript $\dagger$ denotes the Moore-Penrose pseudo-inverse, the authors proceed to introduce centrality and communicability indices in terms of diagonal and off-diagonal elements of this matrix exponential; we refer the interested reader to~\cite{MR3005305} for more details.
By referring back to \cref{tab:closed_forms}, it is easy to see that the diagonal blocks rewrite as $E_{2}(AA^T)$ and $E_{2}(A^TA)$, respectively. 
As for the off-diagonal blocks, these as well can be written using the generalized matrix function  induced by $E_{2,2}(z)$; see~\cite{arrigo2016computation,aurentz2019stable,hawkins1973generalized,noferini2017formula} for a complete discussion of generalized matrix functions and their computation. 

To the best of our knowledge, the ML function $E_{1,2}(z)$ has appeared at least twice in the network science literature. The first appearance is in a paper by Estrada~\cite{estrada2010generalized}, where entries of the matrix function $\psi_1(A) = A^{-1} (e^A-I) = E_{1,2}(A)$ are used as a centrality measure for the nodes of an undirected graph represented by the invertible matrix $A$. 
\begin{remark}\label{rem:singularity}
We note in passing that $E_{1,2}(z)=\psi_1(z) = \sum_{r=0}^\infty \frac{z^r}{(r+k-1)!}$ is entire and thus, by \cite[Theorem  4.7]{HigFM}, the matrix function $E_{1,2}(A)=\psi_1(A)$ is defined and given by $\psi_1(A)= \sum_{r=0}^\infty \frac{A^r}{(r+k-1)!}$ even for singular matrices. 
\end{remark} 
In the same paper, the author actually introduces a larger family of measures, all defined in terms of the functions $\psi_{k-1}(z) = E_{1,k}(z)$ for $k=2,3,\ldots$. As in Remark~\ref{rem:singularity}, care should be taken when working with the induced matrix function: the power series expression is well-defined, while the form $A^{1-k}(e^A-\sum_{r=0}^{k-2}A^r)$ is only defined for invertible matrices. 

A second appearance of the matrix function induced by $E_{1,2}(z)=\psi_1(z)$ is in~\cite{arrigo2018exponential}, where the authors show that the non-backtracking exponential generating function for simple graphs is:
\[
\sum_{r=0}^\infty \frac{p_r(A)}{r!} = 
\begin{bmatrix} I & 0\end{bmatrix} \psi_1(Y) 
\begin{bmatrix} A \\A^2-D\end{bmatrix} + I,
\]
where $p_r(A)$ is a matrix whose entries represent the number of non-backtracking walks of length $r$ between any two given nodes, $D$ is the degree matrix, and $Y$ is the first companion linearization of the matrix polynomial $(D-I) -A\lambda +I\lambda^2$:
\[
Y = 
\begin{bmatrix}
0 & I \\
I-D & A
\end{bmatrix}
;\] see~\cite{arrigo2018exponential} for more details and for the discussion of the directed case. 

Yet another instance of Mittag--Leffler function can be found in~\cite{estrada2017accounting} (and more recently in~\cite{estrada2020topological}), where the authors introduce new centrality and communicability indices by exploiting entries of the matrix function induced by
\[
f(z) = \sum_{r=0}^\infty \frac{z^r}{r!!} = \frac{1}{2}\left[\sqrt{2\pi} \,\text{erf}\left(\frac{z}{\sqrt{2}}\right)+2\right]e^{z^2/2},\qquad 
    r!! = \prod_{k=0}^{\lceil \frac{r}{2}\rceil} (r-2k),
\]
which, after a simple manipulation, rewrites as: 
\[
f(z) = \sqrt{\frac{\pi}{2}} E_{1/2}(z/\sqrt{2}) + \left(\sqrt{\frac{\pi}{2}} + 1 \right)E_{1}(z^2/2).
\]

More recently, the matrix function induced by $E_{1/2}(z)$ was used in~\cite{abadias2020fractional} to describe a model for the transmission of perturbations across the amino acids of a protein represented as an interaction network. 

\smallskip

In the following subsection, we discuss two key points concerning interpretation and computability of the matrix functions induced by $E_\alpha(z)$.

\subsection{Parameter selection}
We want to discuss in this section a few technicalities that should be kept in mind when working with Mittag--Leffler functions. 
We discuss two main points: the first concerns the monotonicity of the coefficients (as a function of $r$) appearing in the power series expansion~\eqref{eq:MLfunction} defining $E_\alpha(z)$. This will motivate the use of parametric ML functions $E_\alpha(\gamma z)$ when defining network indices. Secondly, we will discuss issues related to the representability of the entries of $E_\alpha(\gamma A)$ for large matrices and, more generally, for matrices with a large leading eigenvalue. 

We begin by discussing the monotonicity of the coefficients in the power series expansion~\eqref{eq:MLfunction} defining $E_{\alpha}(z)$. 
As previously mentioned in \cref{sec:themlfunction}, the function $g(r):=\Gamma(\alpha r+1)$ is not monotonic for certain values of $\alpha\in(0,1)$; see \cref{fig:coeff_gamma}.  
An immediate consequence of this in our framework is that the matrix function
\[
E_\alpha(A) = \sum_{r=0}^\infty \frac{A}{\Gamma(\alpha r+1)}
\]
is no longer weighting walks monotonically depending on their length. For example, when $\alpha = 0.8$ walks of length one are weighted by the coefficient $c_1(0.8) \approx 0.9$, whilst walks of length five have weight $c_5(0.8) = 24$. 
We want to stress that this may not be an issue in certain application; however, it is usually the case in network science that walks are assigned monotonically decreasing weights with their lengths. 

Let us thus consider the following parametric ML function:
\[
\widetilde{E}_\alpha(z) = E_\alpha(\gamma z) = \sum_{r=0}^\infty \frac{(\gamma z)^r}{\Gamma(\alpha r+1)} = \sum_{r=0}^\infty \widetilde{c}_r(\alpha,\gamma) z^r
\]
where $\widetilde{c}_r(\alpha,\gamma) = \gamma^r c_r(\alpha)$, for suitable values of the weight $\gamma>0$. The next Lemma provides conditions on the admissible vales of $\gamma$ to ensure monotonicity of the coefficients $\widetilde{c}_r(\alpha,\gamma)$.

	\begin{lemma}\label{lem:monotoniccoeffs} 
	Suppose that $\alpha \in (0,1)$. The coefficients $\widetilde{c}_r(\alpha,\gamma) = \gamma^r c_r(\alpha)$ defining the power series for the entire function $\widetilde{E}_\alpha(z) = E_{\alpha}(\gamma z)$ are monotonically decreasing as a function of $r = 0,1,2,\ldots$ for all $0<\gamma<\Gamma(\alpha+1)$.
	\end{lemma}
    \begin{proof}
    For each $\alpha\in(0,1)$ we want to determine conditions on $\gamma = \gamma(\alpha)$ that imply that 
    \[\Tilde{c}_r(\alpha,\gamma) \geq \Tilde{c}_{r+1}(\alpha,\gamma) \quad \text{for all } r\in\mathbb{N} \]
    From the definition of $\Tilde{c}_r(\alpha,\gamma)$ we have that the above inequality is equivalent to verifying 
    \begin{equation*}\label{eq:cond_gamma}
    \gamma \leq \frac{\Gamma (\alpha r+\alpha+1)}{\Gamma (\alpha r+1)}, \quad \text{for all } r\geq 0 
    \end{equation*}
    since $\gamma>0$ and $\Gamma(x)>0$ for all $x\geq 0$. 
Since    $H_x$, the \emph{Harmonic number} for $x\in\mathbb{R}$, is an increasing function of $x$, $\alpha>0$ by hypothesis, and $\Gamma(x)>0$ for all $x\geq 0$, it follows that 
\[    \frac{d}{dx}\left(\frac{\Gamma (\alpha x+\alpha+1)}{\Gamma (\alpha x+1)}\right)  = \frac{\alpha \left(H_{ \alpha(x+1)}-H_{\alpha x}\right) \Gamma (\alpha x+\alpha+1)}{\Gamma (\alpha x+1)}\geq 0,\]
and thus the minimum of $\frac{\Gamma (\alpha x+\alpha+1)}{\Gamma (\alpha x+1)}$ is achieved at $x=0$.
\end{proof}

Two choices of the parameter $\alpha$ require further discussion. Suppose that $A\in\mathbb{R}^{n\times n}$ is the adjacency matrix of a simple non-empty graph. 
\begin{itemize}
\item When $\alpha = 0$, then $E_0(\gamma A) = (I -\gamma A)^{-1}$ admits a convergent series expansion if and only if $|\gamma\lambda|<1$ for all $\lambda$ eigenvalues of $A$. The coefficients of this expansion are $\gamma^r$, which are decreasing for all the admissible $0<\gamma\leq \rho(A)^{-1}$.
\item When $\alpha=1$, then $E_1(\gamma A) = {\rm exp}(\gamma A)$ and the coefficients $\gamma^r/r!$ are decreasing for $0<\gamma\leq 1$.
\end{itemize}
In \cref{fig:admissible_beta} we display the area of admissible choices of $\gamma$ as a function of $\alpha\in(0,1]$. 

\begin{figure}[t]
    \centering
    \includegraphics[width=0.4\textwidth]{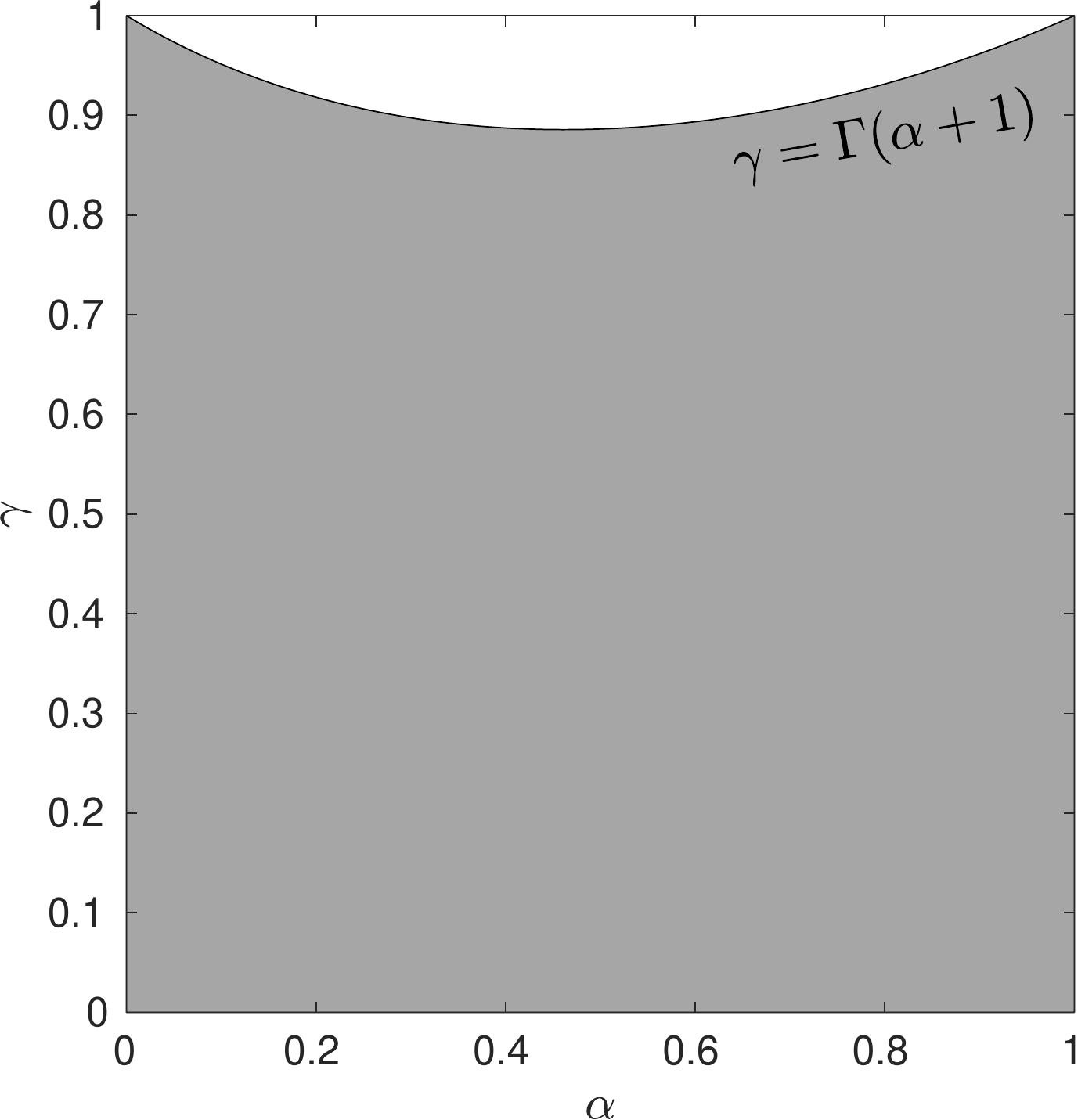}
    \caption{Admissible values of $\gamma$ as a function of $\alpha\in(0,1]$.}
    \label{fig:admissible_beta}
\end{figure}
	
	The take home message of \cref{lem:monotoniccoeffs} wants to be that Mittag--Leffler functions with $\alpha\in (0,1)$ can be employed in network science problems since they have a power series expansion that can be interpreted in terms of walks; however, care should be taken since the coefficients of the ML may not have the desired monotonic behavior.
    In particular, the choice $\gamma = 1$ is not always viable, since it yields non-monotonically decreasing coefficients $c_r(\alpha)$ for those values of $\alpha\in(0,1]$ that satisfy $\Gamma(\alpha+1)<1$, i.e., for all $\alpha\neq 1$. 

\bigskip	

	The second point that we want to address is when the magnitude of the entries of the matrix function $E_\alpha(\gamma A)$ exceeds the largest representable number in machine precision. Consider the spectral decomposition of the adjacency matrix $A = Q\Lambda Q^T$. Then, by definition of matrix function
$E_{\alpha,1}(\gamma A) = \gamma\ Q E_{\alpha}( \Lambda) Q^T$. 	For matrices such that  $\lambda_{\max}(A)$ is large enough, $\gamma E_{\alpha}(\lambda_{\max}(A))$ may be larger than the largest representable number $\Bar{N}$ in machine precision. 
		The following result details the constraint on the values of $\gamma\in (0,1]$ which ensures represenatability of   $E_{\alpha,1}(\gamma\lambda_{\max}(A))$. 
	\begin{lemma}\label{lemma:representable}
	Suppose that $\alpha \in (0,1]$, and $A  \in \mathbb{R}^{n \times n}$ is symmetric. Then for all 
	\[
	\gamma \leq \frac{1}{\lambda_{\max}(A)}\left(\bar{K}\log(10) + \log(\alpha)\right)^\alpha
	\]
	it holds that $\max_{i,j}(|E_{\alpha}(\gamma A)|)_{i,j} \leq \Bar{N}$ where $\Bar{N} \approx 10^{\Bar{K}}$ for a given $\bar{K} \in \mathbb{N}$ is the largest representable number on a given machine.
	\end{lemma}
		
		Before proceeding with the proof, let us recall the following result, which describes an asymptotic expansions for ML functions.

	\begin{proposition}~\cite[Proposition~3.6]{MR3244285}\label{prop:MLasympt}
	Let $0<\alpha<2$ and $\theta \in (\frac{\pi \alpha}{2},\min(\pi,\alpha\pi))$. Then we have the following asymptotics for the Mittag--Leffler function for any~$p \in \mathbb{N}$
	\begin{eqnarray*}
	E_{\alpha}(z) = \frac{1}{\alpha} e^{z^{\frac{1}{\alpha}}} - \sum_{k=1}^{p} \frac{z^{-k}}{\Gamma(1-\alpha k)} + O(\lvert z\rvert^{-1-p}), \, \lvert z \rvert \rightarrow + \infty, \, \lvert \arg(z)\rvert \leq \theta,\\
	E_{\alpha}(z) = - \sum_{k=1}^{p} \frac{z^{-k}}{\Gamma(1-\alpha k)} + O(\lvert z\rvert^{-1-p}), \, \lvert z \rvert \rightarrow + \infty, \, \theta \leq \lvert \arg(z)\rvert \leq \pi.\nonumber
	\end{eqnarray*}
	\end{proposition}

	\begin{proof}[Proof of \cref{lemma:representable}]
	We have $\lambda_{\max}(\gamma A) =\gamma \lambda_{\max}(A) \in \mathbb{R}$, since $A$ is symmetric; then by \cref{prop:MLasympt}, using the fact that $\arg(z) = 0$ for $z \in \mathbb{R}$, for $p = 0$ we find   
	\begin{equation*}\label{eq:cond_coeff2}
	\frac{1}{\alpha} e^{(\gamma\lambda_{\max}(A))^{\frac{1}{\alpha}}} \leq \bar{N} \approx 10^{\bar{K}}, 
	\end{equation*}
    which immediately yields the conclusion.
	\end{proof}

	Combining the results of \cref{lem:monotoniccoeffs} and \cref{lemma:representable}, we can thus provide the following result which summarizes viable choices of the parameter $\gamma$ for a given choice of $\alpha\in (0,1)$. 
	\begin{proposition}\label{pro:selecting_b}
	Let $A$ be the adjacency matrix of an undirected network with at least one edge and let $\rho(A)>0$ be its spectral radius. Moreover, let $\Bar{N}\approx 10^{\Bar{K}}$ be the largest representable number on a given machine. Then the  Mittag--Leffler function $\widetilde{E}_\alpha(z) = E_{\alpha}(\gamma z)$ is representable in the machine, and admits a series expansion with decreasing coefficients when $\alpha \in (0,1)$ and 
	\begin{equation}\label{eq:m}
0< \gamma \leq  \mu(\alpha):=\min\left\lbrace \Gamma(\alpha+1), \frac{\left( \bar{K}\log(10) + \log(\alpha)\right)^\alpha}{\rho(A)}\right\rbrace.
\end{equation}
	
	    \end{proposition}

\subsection{Mittag--Leffler network indices}
    In this subsection we define centrality indices in terms of functions of the adjacency matrix induced by ML functions. Similarly, communicability indices defined in terms of the off-diagonal entries of the relevant matrix functions can also be introduced.

\begin{definition}\label{def:measure}
Let $A$ be the adjacency matrix of a simple graph $G = (V, E)$. Let $\alpha\in[0,1]$ and let $0<\gamma<\Gamma(\alpha+1)$, so that \cref{lem:monotoniccoeffs} holds. Then, for all nodes $i\in V = \{1,2,\ldots,n\}$ we define: 
    \begin{itemize}
        \item ML-subgraph centrality: \[s_i(\widetilde{E}_\alpha) = E_\alpha(\gamma A)_{ii}\]
        \item ML-total communicability: \[t_i(\widetilde{E}_\alpha) = (E_\alpha(\gamma A)\bone)_{i}\]
    \end{itemize}
\end{definition}

Since $\gamma$ satisfys the hypothesis of \cref{lem:monotoniccoeffs}, the coefficients $\frac{\gamma^r}{\Gamma(\alpha r+1)}$ in the power series representation of $E_\alpha(\gamma A)$ are monotonically decreasing. We can thus interpret the entries of this matrix function as a weighted sum of the number of walks taking place in the network with longer walks being given less weight than shorter ones.
    \begin{remark}
    Similarly, an index of subgraph communicability can be defined as $C_{ij}(\widetilde{E}_\alpha) = E_\alpha(\gamma A)_{ij}$ for all $i,j\in V$, $i\neq j$.
    \end{remark}

 These centrality %
 indices arise as a straightforward extension of known theory for undirected graphs, namely the exponential-based subgraph centrality and total communicability and their resolvent-based analogues.
	The newly introduced indices all belong to the class of indices studied in \cite{MR3354998}; Indeed, it can be easily shown that the rankings induced by the subgraph centrality and total communicability indices $\svec(\widetilde{E}_\alpha(A))$ and $\tvec(\widetilde{E}_\alpha(A))$ converge to those induced by degree and eigenvector centrality as $\gamma \to 0$ and  as $\gamma\to\infty$ (or $\gamma\to\rho(A)^{-1}$, when $\alpha=0$), respectively. It is worth mentioning that the upper limit considered here is the same as it was considered in \cite{MR3354998}, although the results presented in \cref{pro:selecting_b} provide a different upper bound on the admissible values for $\gamma$. 

It can be further shown that the measures here introduced converge to those induced by the exponential and by the resolvent as we keep the value of $\gamma$ fixed and we let the parameter $\alpha$ vary in the interval $(0,1)$. 
Here, the convergence is actually shown for the centrality scores, rather than just for the induced ranking. 
Indeed, suppose that $\gamma<\min\left\lbrace \Gamma(\alpha+1), 1/\rho(A)\right\rbrace$, so that the power series expansion for $E_\alpha(\gamma A)$ converges for all values of $\alpha$ and the coefficients appearing in said series are monotonically decreasing. 
Then it is straightforward to show that,
\begin{itemize}
\item for $f(z) = (1-\gamma z)^{-1}$,
\[\lim_{\alpha \to 0}\svec(\widetilde{E}_\alpha) = \svec(f)\,\text{ and }\, \lim_{\alpha \to 0}\tvec(\widetilde{E}_\alpha) = \tvec(f);\]
\item for $f(z) = e^{\gamma z}$
\[\lim_{\alpha \to 1}\svec(\widetilde{E}_\alpha) = \svec(f)\,\text{ and }\,\lim_{\alpha \to 1}\tvec(\widetilde{E}_\alpha) = \tvec(f).\]
\end{itemize}
\cref{fig:convergecebehavior} schematically summarizes these results. 
	
	\begin{figure}[t]
		\centering
		\includegraphics[width=0.7\columnwidth]{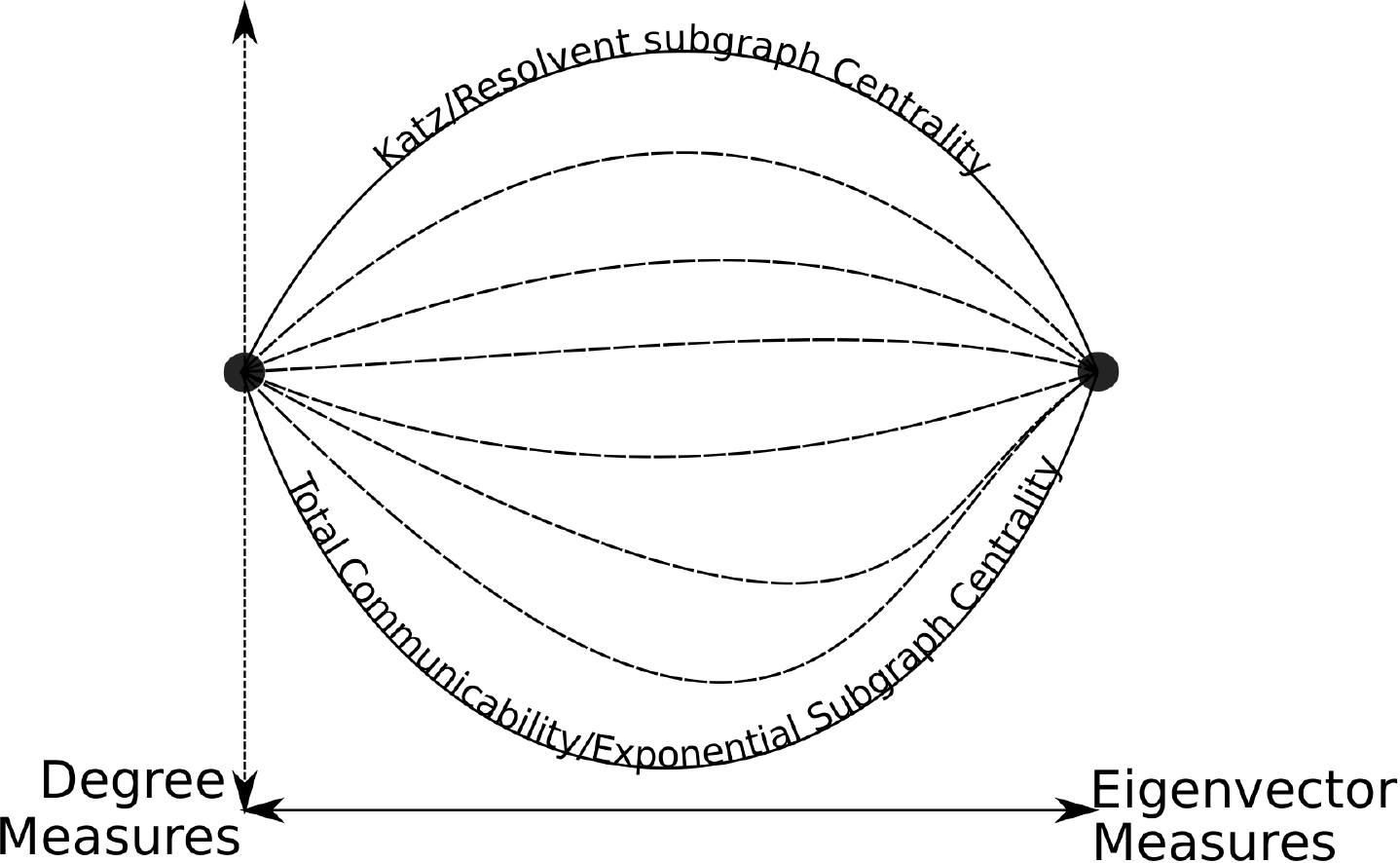}
		\caption{Asymptotic behavior of the new Mittag--Lefller based measures.}\label{fig:convergecebehavior}
	\end{figure}

\subsection{Computing the ML function} 
The computation of the ML function $E_{\alpha,\beta}(z)$ is far from being straightforward. Indeed, for different parts of the complex plane one has to advocate different numerical techniques with different degrees of accuracy. Furthermore, when one wants to compute the induced matrix function for the case of non-normal matrices, the derivatives of arbitrary order also need to be computed. 
However, for particular choices of the parameters $\alpha$ and $\beta$ we could employ specialized techniques; for example, when $\alpha=\beta=1$ the ML function reduces to the exponential, and for this matrix function there are several techiniques available in the literature; see, e.g.,~\cite{moler1978nineteen} and references therein.  
In this paper we are faced with the problem of computing $E_\alpha(z)$ for arbitrary choices of $\alpha$. To accomplish this task, we use the techniques and the code developed in~\cite{MR3850348}. Furthermore, to compute the total communicability $\mathbf{t}(E_{\alpha})$ we deploy such approach in a standard polynomial Krylov method. 
In a nutshell, we are projecting the problem of computing $E_{\alpha,\beta}(\gamma A)\mathbf{v}$ on the subspace $\mathcal{K}_m(A,\mathbf{v}) = \{\textbf{v},A\textbf{v},\ldots,A^{m-1}\textbf{v}\}$, that is, we  compute the approximation $E_{\alpha,\beta}(\gamma A)\mathbf{v} \approx V_m E_{\alpha,\beta}(\gamma V_m^T A V_m )\textbf{e}_1$, where $V_m = [\textbf{v}_1,\ldots,\textbf{v}_m]$ is a basis of $\mathcal{K}_m(A,\textbf{v})$, and $\textbf{e}_1$ the first vector of the canonical basis of $\mathbb{R}^m$.
For an analysis of the convergence of such method we refer the interested reader to \cite[Theorem~3.7]{MoreNovati-Krylov}. 
In fact, one could also employ rational Krylov methods pursuing a  trade-off between the size of the projection subspace and the cost of the construction of the basis $V_m$. For the analysis of this other approach, please see \cite{MoreNovati-Krylov, MoretPopolizio-SAIML}. 

In the experiments presented in this paper, as mentioned above, we considered  polynomial methods, which  already gave satisfactory performances. 

\subsection{Numerical experiments - centrality measures}

In this section we explore numerically how the measures introduced in \cref{def:measure} compare with eigenvector centrality and degree centrality as we let $\alpha$ and $\gamma$ vary. To make the comparison, we use of the Kendall correlation coefficient~\cite{kendall1948rank}: the higher the coefficient, the stronger the correlation. 
The networks analysed here are two networks freely available at~\cite{SUITESPARSE}. 
The network {\sc Newman/Dolphins} contains $n=62$ nodes and $m = 139$ undirected edges. Its largest eigenvalues are $\lambda_1=   7.19$ and $\lambda_2 =  5.94$.
The network {\sc Gleich/Minnesota} contains $n=2640$ nodes and $m=3302$ undirected edges. Its largest eigenvalues are $\lambda_1 =  3.2324$ and $\lambda_2= 3.2319$, and therefore this network has a relatively small spectral gap $\lambda_1-\lambda_2$. 
Results are displayed in \cref{fig:kendall_sc_dolphins,fig:kendall_tc_dolphins,fig:kendall_sc_minnesota,fig:kendall_tc_minnesota}. In these figures we also plot a solid line to display the value of $\mu(\alpha)$ in \cref{eq:m}: this provides an upper bound on the admissible values of $\gamma$. We note in passing that the function accurately profiles the NaN region in each of our plots, which corresponds to values of $\alpha$ and $\gamma$ for which the computed measures exceeded machine precision. 

In \cref{fig:kendall_sc_dolphins}-\ref{fig:kendall_tc_dolphins}, we observe that, after the maximum of $\mu(\alpha)$, the  correlation of the newly computed measure with eigenvector centrality (\cref{fig:kendall_tc_eig_dolphins} and \cref{fig:kendall_sc_eig_dolphins}) increases as $\alpha$ increases and $\gamma$ increases, even above the curve $\mu(\alpha)$. This demonstrates the known fact, proved in~\cite{MR3354998}, that ML functions induce centrality measures that provide the same ranking as eigenvector centrality when $\gamma\to\infty$.  
In \cref{fig:kendall_sc_deg_dolphins} and \cref{fig:kendall_tc_deg_dolphins}, on the other hand, we achieve larger values of the Kendall $\tau$ for small values of $\gamma$, regardless of the value of $\alpha$, as expected.

Similar results were achieved for the network {\sc Gleich/Minnesota} in  \cref{fig:kendall_sc_minnesota}-\ref{fig:kendall_tc_minnesota},  although not strong correlation is observed between the new indices and eigenvector centrality. This is again a known result, and it is due to the small spectral gap of the adjacency matrix of this network. 
For this graph it is however interesting to note the high degree of correlation between the new measure and eigenvector centrality for small values of $\alpha$. 
 {\begin{remark} We visually inspected a few of the top ranked nodes according to degree centrality, eigenvector centrality, and ML-subgraph centrality and ML-total communicability for different values of $\alpha$ and $\gamma$ (ten nodes for {\sc Gleich/Minnesota} and $20\%$ of the total number of nodes for {\sc Newman/Dolphins}). 
We can confirm that the ML measures, where well defined, return results comparable with those presented for the whole network when working on {\sc Newman/Dolphins}. The results are not as good for {\sc Gleich/Minnesota}, as one would expect because of the network's spectral gap. 
We refer the interested  reader to the Supplementary Material for further details.\end{remark}}

One interesting feature of all these plots is that the centrality measures studied seem to strongly correlate with either degree centrality or eigenvector centrality, with only a small interval of values of $\gamma$ for each $\alpha$ where the correlation is not strong. This in particular has implications when we consider the two most popular Mittag--Leffler functions used in the literature: $e^{\gamma x}$ and $(1-\gamma x)^{-1}$. Indeed, this result shows that most of the choices of $\gamma$, the downweighting parameter (for resolvent-based measures) or inverse temperature (for exponential-based measures), return rankings that can be obtained by simply computing the eigenvector centrality or the degree centrality of the network. However, if one can hit the ``sweet spot", with values of $\alpha$ and $\gamma$ that return centralities not strongly correlated with the two classical ones, using these measures will certainly add value to the analysis. A similar observation was made in the Supplementary material of~\cite{MR3354998}, where the authors write:{\it ``Thus, the most information is gained by using resolvent based centrality measures when $0.5/\rho(A)\leq \gamma\leq 0.9/\rho(A)$. This supports the intuition from section 5 of the accompanying paper that “moderate” values of $\gamma$ provide the most additional information about node ranking beyond that provided by degree and eigenvector centrality"}. We plan to investigate this phenomenon further and to describe ways to select $\gamma$ for each value of $\alpha$ in future work. We note that some work in a similar direction was conducted in \cite{aprahamian2016matching}.

\begin{figure}[tbhp]
\centering
\subfloat[]{\label{fig:kendall_sc_deg_dolphins}\includegraphics[width=0.45\textwidth]{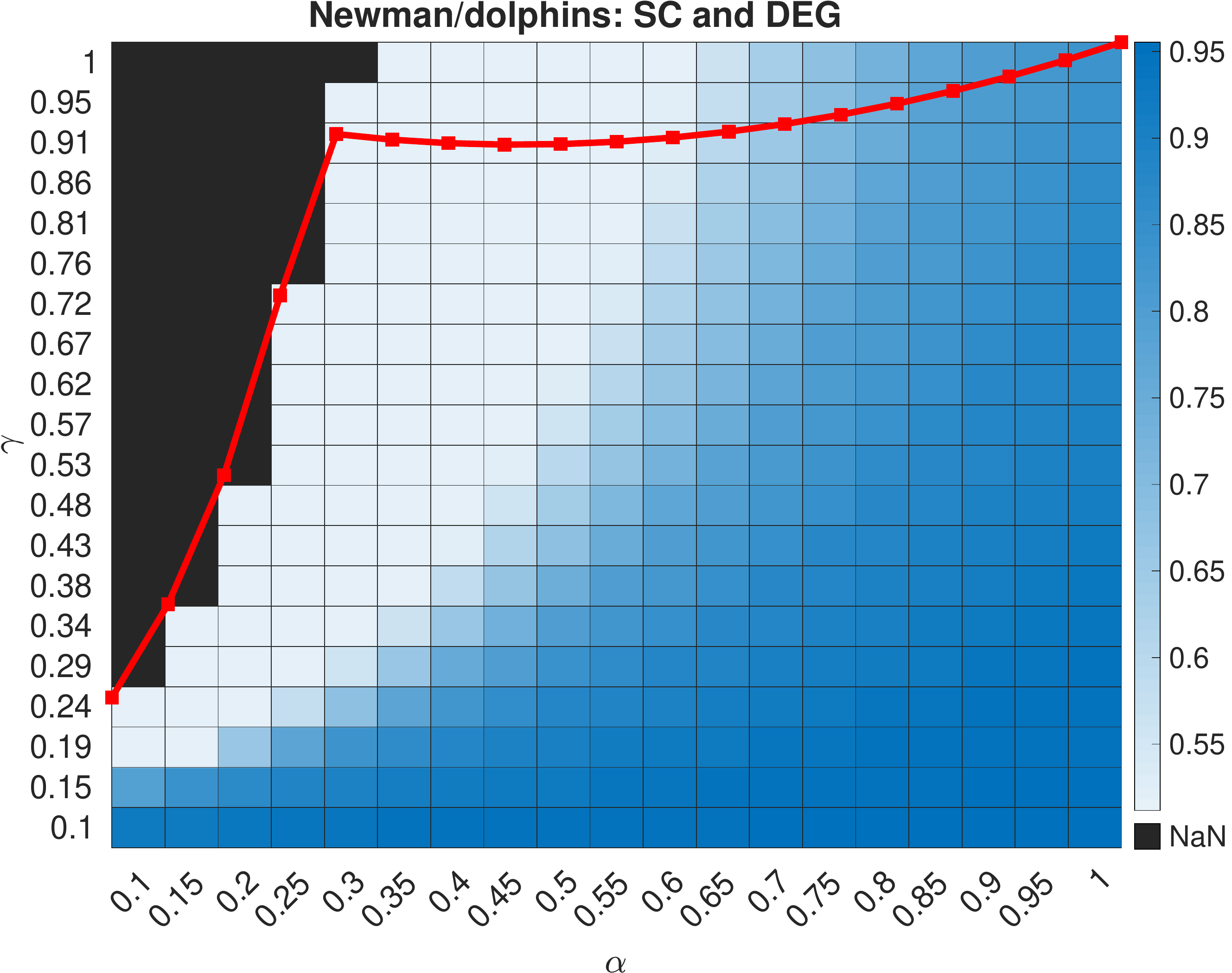}}\;
\subfloat[]{\label{fig:kendall_sc_eig_dolphins}\includegraphics[width=0.45\textwidth]{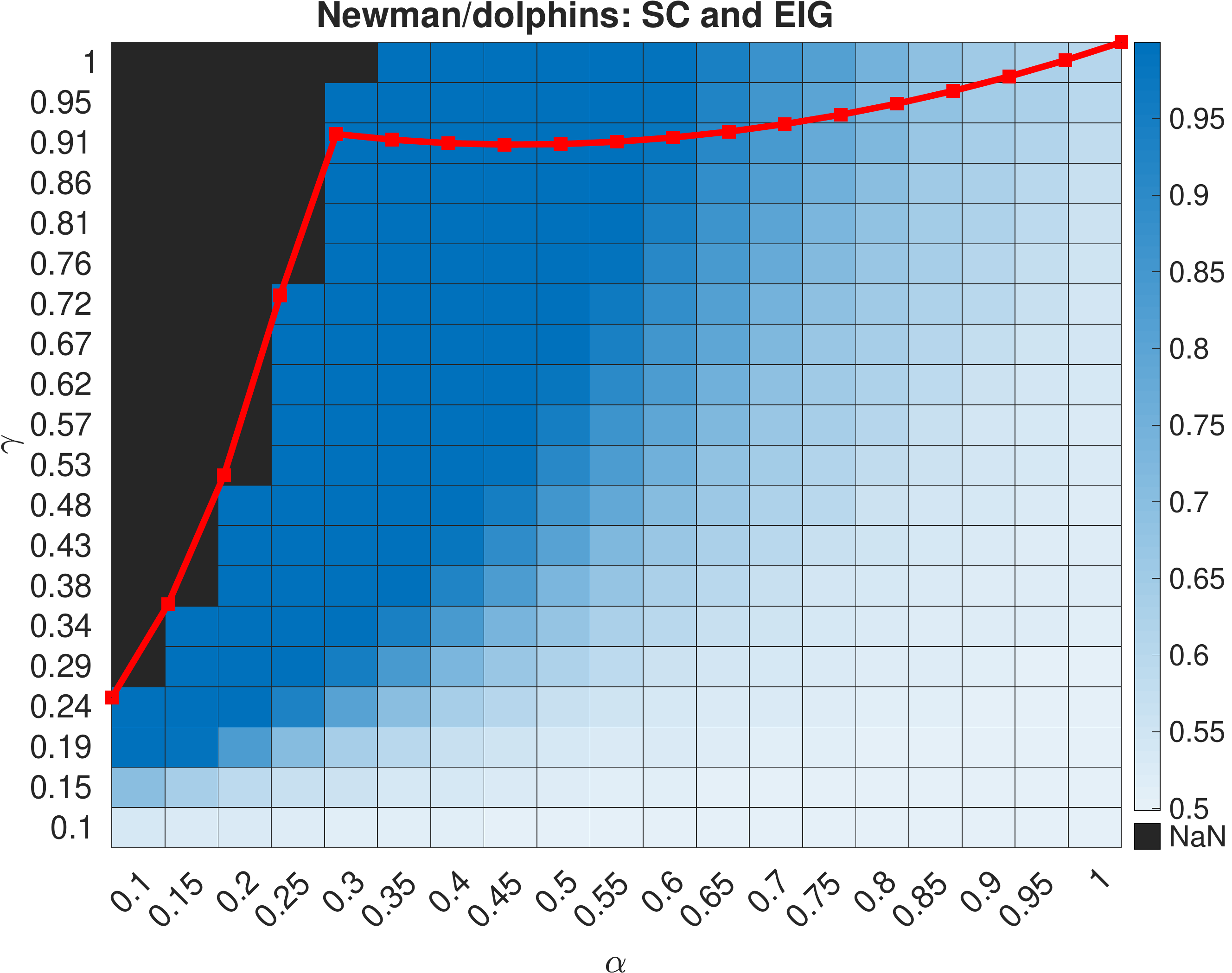}}
\caption{Network: {\sc Newman/Dolphins}. Kendall correlation coefficient between the ranking induced by subgraph centrality vectors $\svec(\widetilde{E}_\alpha)$ and by (a) degree centrality or (b) eigenvector centrality  for different values of $\gamma$ and $\alpha$. The red line displays the value of $\mu$ in \eqref{eq:m}.}
\label{fig:kendall_sc_dolphins}
\end{figure}

\begin{figure}[tbhp]
\centering
\subfloat[]{\label{fig:kendall_tc_deg_dolphins}\includegraphics[width=0.45\textwidth]{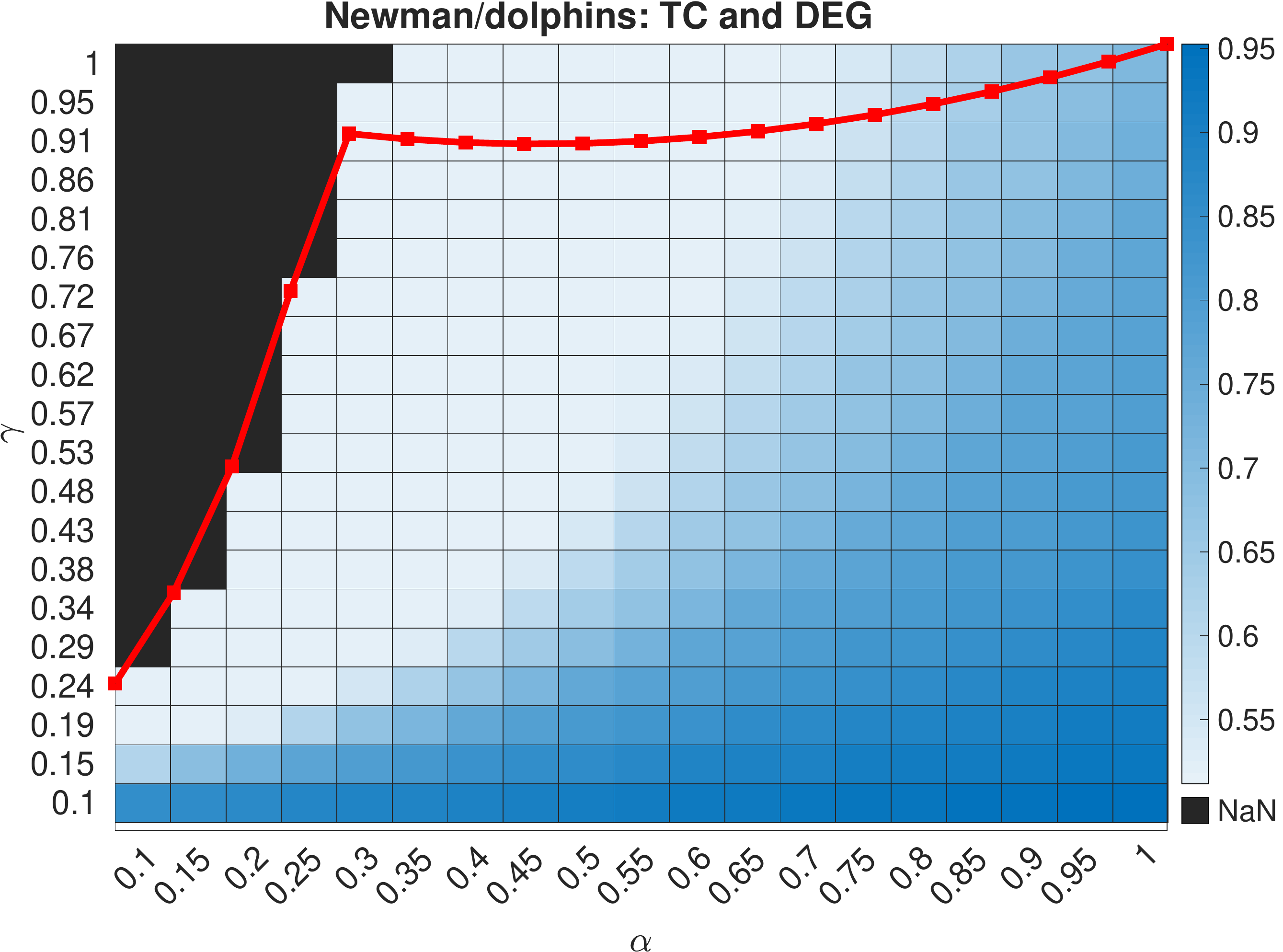}}\;
\subfloat[]{\label{fig:kendall_tc_eig_dolphins}\includegraphics[width=0.45\textwidth]{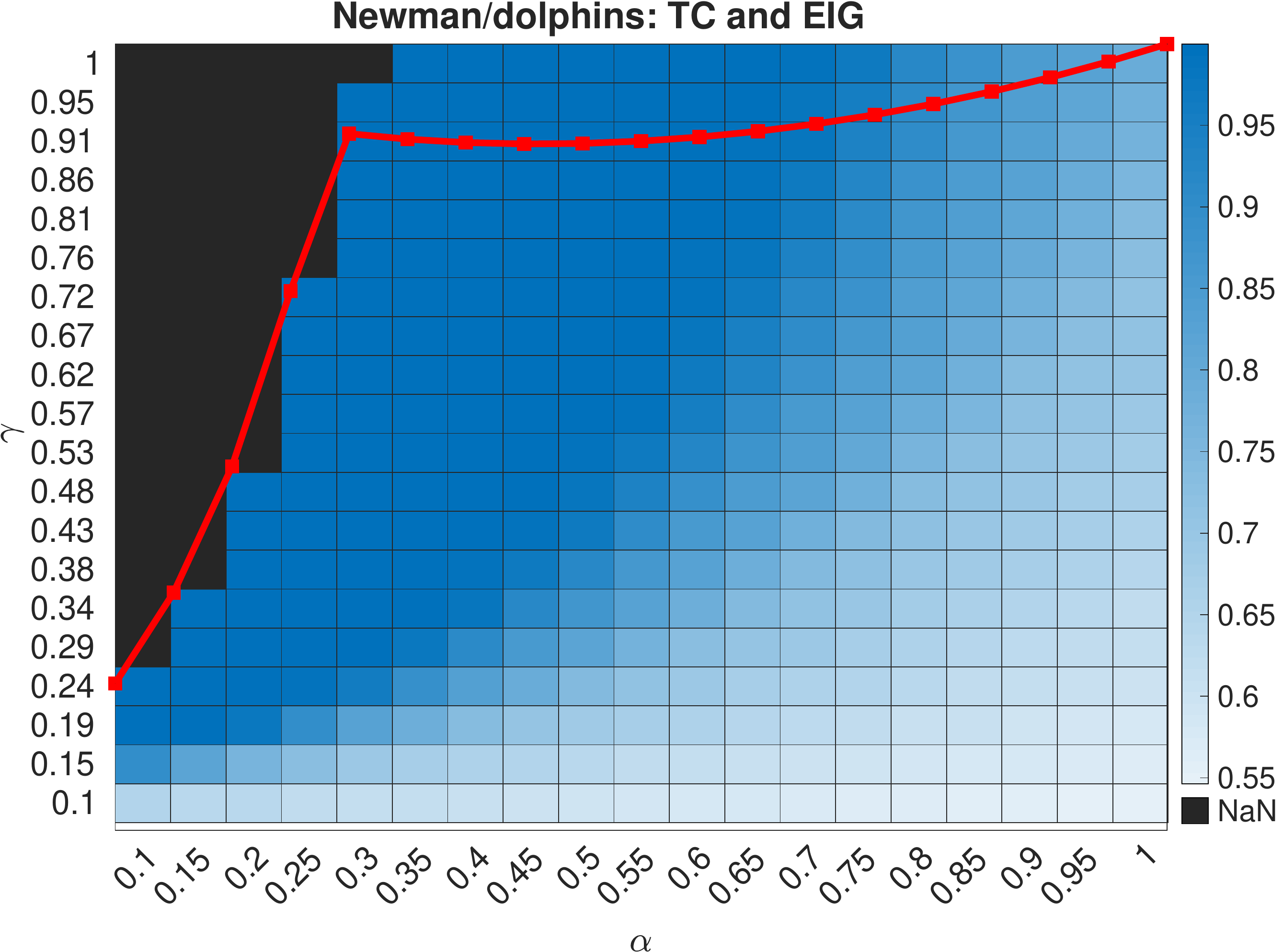}}
\caption{Network: {\sc Newman/Dolphins}. Kendall correlation coefficient between the ranking induced by total communicability vectors $\tvec(\widetilde{E}_\alpha)$ and by (a) degree centrality or (b) eigenvector centrality  for different values of $\gamma$ and $\alpha$. The red line displays the value of $\mu$ in \eqref{eq:m}.}
\label{fig:kendall_tc_dolphins}
\end{figure}

\begin{figure}[tbhp]
\centering
\subfloat[]{\label{fig:kendall_sc_deg_minnesota}\includegraphics[width=0.45\textwidth]{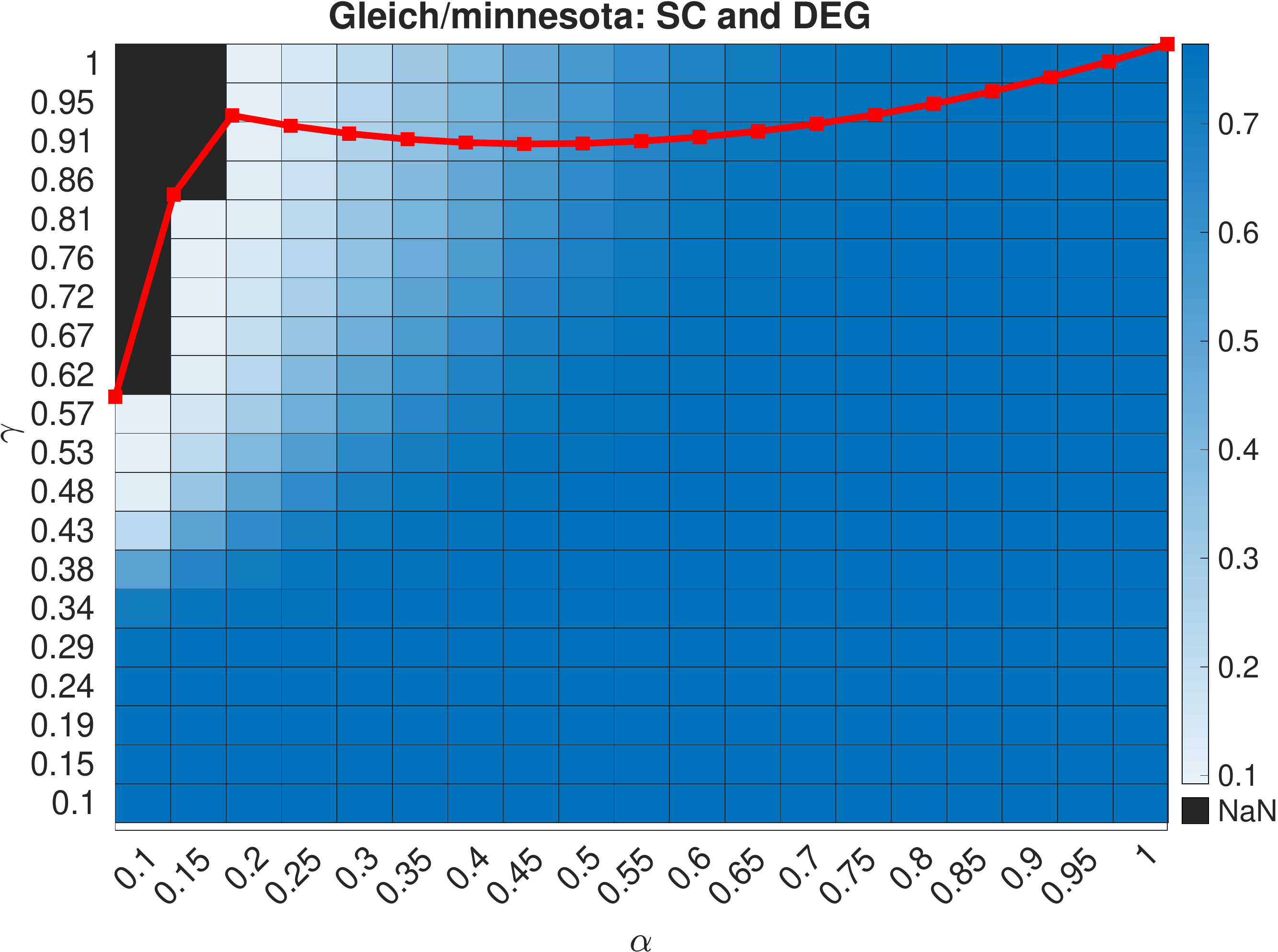}}\;
\subfloat[]{\label{fig:kendall_sc_eig_minnesota}\includegraphics[width=0.45\textwidth]{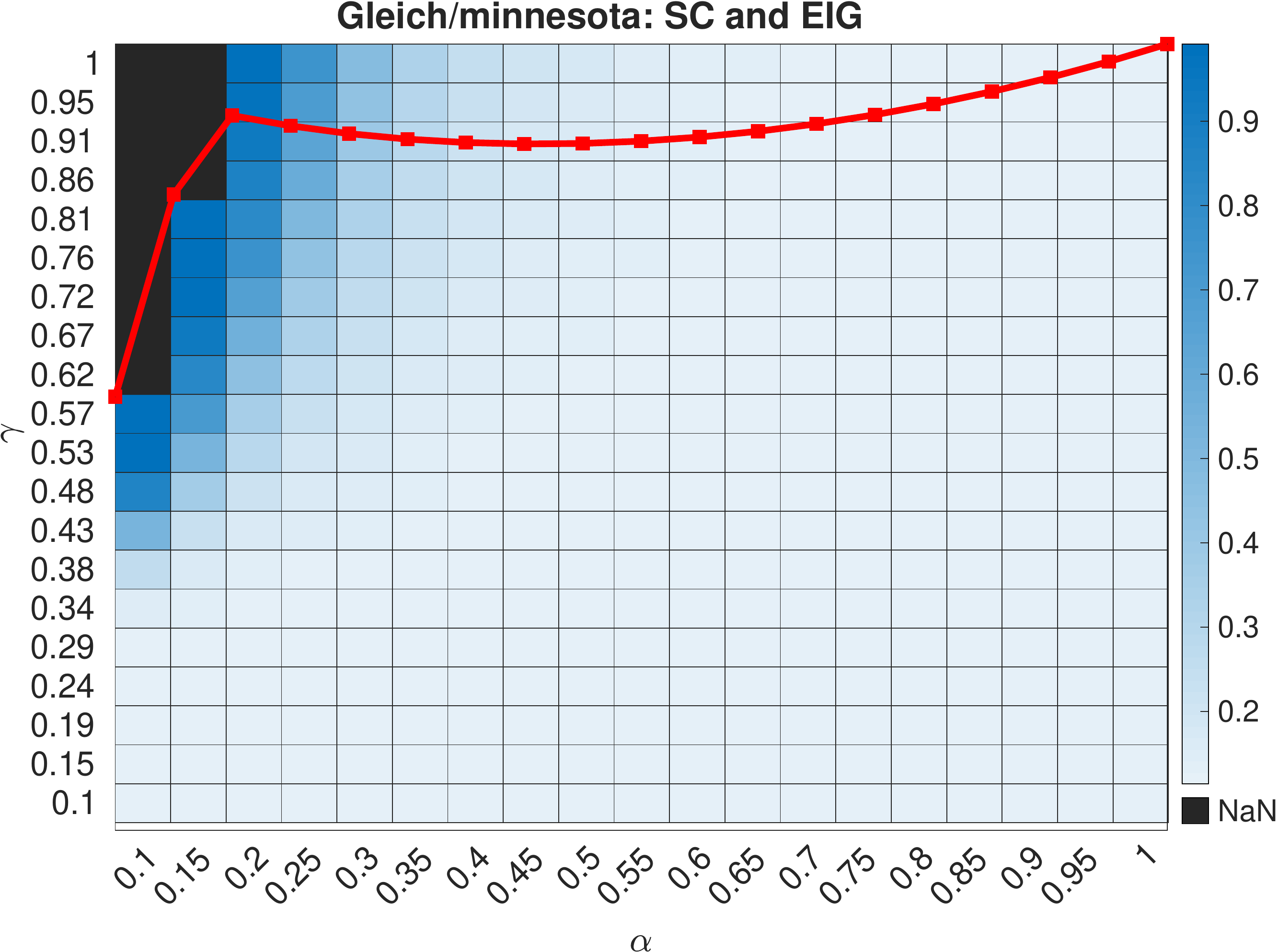}}
\caption{Network:  {\sc Gleich/Minnesota}. Kendall correlation coefficient between the ranking induced by subgraph centrality vectors $\svec(\widetilde{E}_\alpha)$ and by (a) degree centrality or (b) eigenvector centrality  for different values of $\gamma$ and $\alpha$. The red line displays the value of $\mu$ in \eqref{eq:m}.}
\label{fig:kendall_sc_minnesota}
\end{figure}

\begin{figure}[tbhp]
\centering
\subfloat[]{\label{fig:kendall_tc_deg_minnesota}\includegraphics[width=0.45\textwidth]{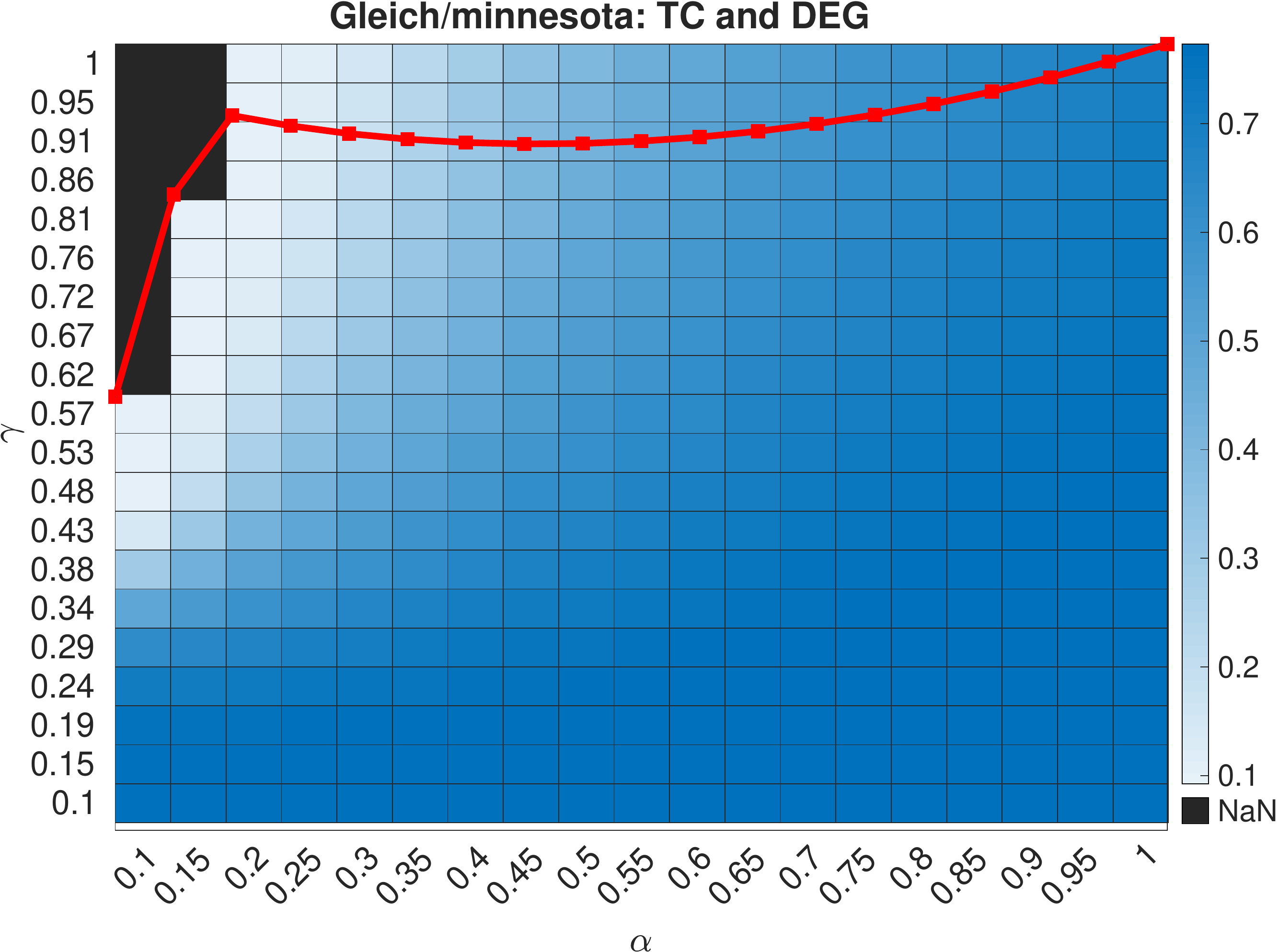}}\;
\subfloat[]{\label{fig:kendall_tc_eig_minnesota}\includegraphics[width=0.45\textwidth]{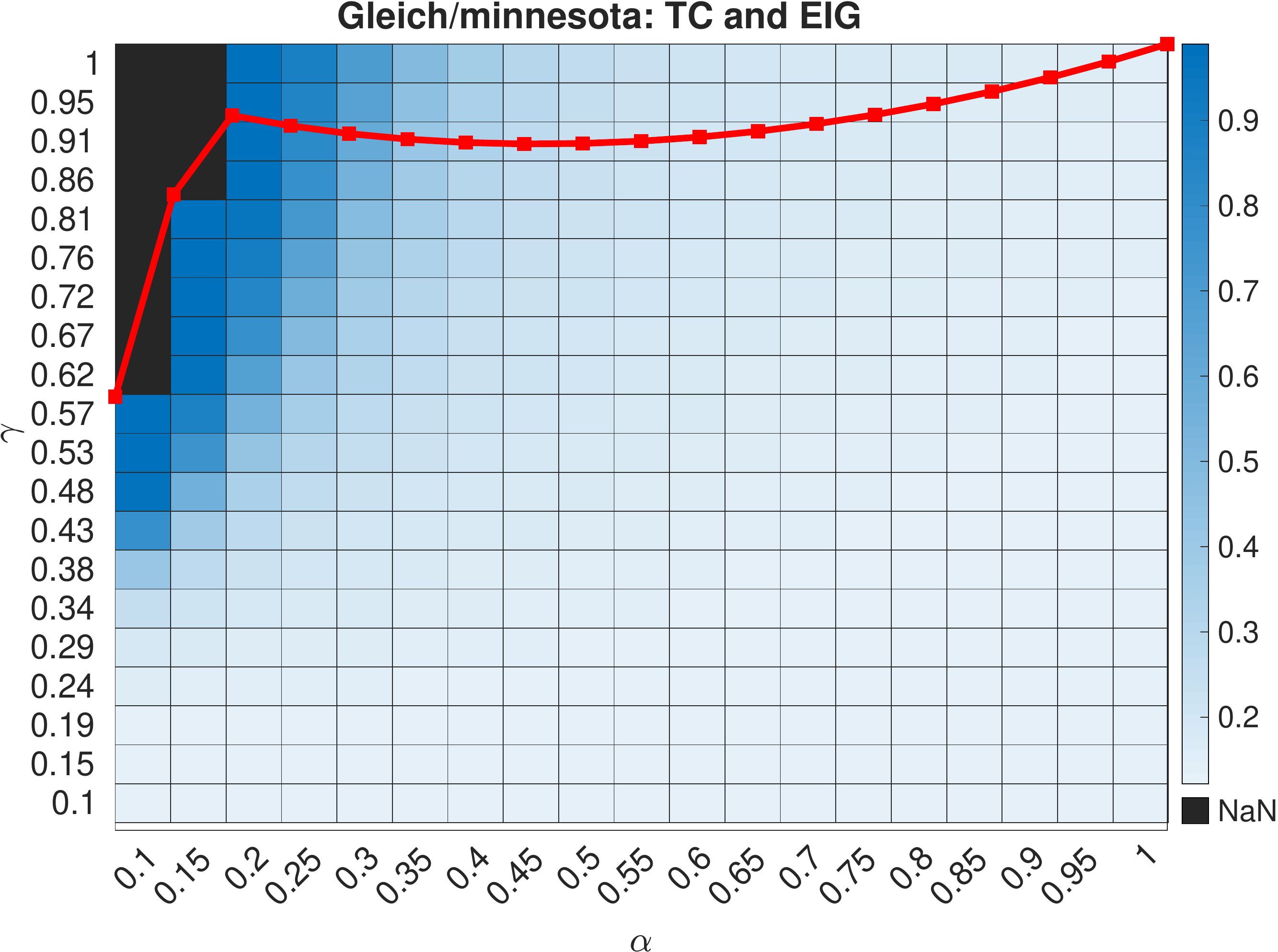}}
\caption{Network:  {\sc Gleich/Minnesota}. Kendall correlation coefficient between the ranking induced by total communicability vectors $\tvec(\widetilde{E}_\alpha)$ and by (a) degree centrality or (b) eigenvector centrality  for different values of $\gamma$ and $\alpha$. The red line displays the value of $\mu$ in \eqref{eq:m}.}
\label{fig:kendall_tc_minnesota}
\end{figure}

\section{Temporal networks}\label{sec:temporal}
Networks are often evolving over time, with edges appearing, disappearing, or changing their weight as time progresses~\cite{holme2012temporal}. Consider a time-dependent network $G = (V,E(t))$ where the nodes remain  unchanged over time, while the edge set $E(t)$ is time-dependent. 
This type of graphs can be described using a time-dependent adjacency matrix $A(t) \, :\, \mathbb{R} \to \mathbb{R}^{n \times n}$, whose regularity depends on the way in which the edges evolve; 
for example, if one wishes to model phenomena characterized by instantaneous activities, then the resulting $t\mapsto A(t)$ will be a discontinuous and rapidly changing function. This model is suited for, e.g., an e-mail communication network, where the different email addresses are the nodes and connections among them are present whenever there is an e-mail exchange between them at a given time $t$~\cite{Stomakhin_2011}. 
On the other hand, suppose that we want to model the number of people entering/leaving a train station. We can assign the value $0$  to the situation where the station is completely empty and value $1$ to the station at full capacity. Then, the function $t\mapsto A(t)$ is at least continuous, and the entries of $A(t)$ take values in $[0,1]$ at all times. 

\smallskip

In the following we will show how the theory of ML functions allows for a generalization of the model presented in~\cite{MR3177230}. This generalization will overcome a known issue of resolvent-based centrality measure for temporal networks: the choice of the downweighting parameter $\gamma$; see \cref{rem:temporal_gamma,rem:temporal_gamma2} below. 
Using ML functions with $\alpha>0$ will automatically free the choice of $\gamma$ from any constraint related to the history of the network, and this parameter will only need to satisfy the conditions prescribed in \cref{pro:selecting_b}.

\smallskip

In~\cite{MR3177230} the authors introduced a real-valued, (possibly) nonsymmetric dynamic communicability matrix $S(t) \in \mathbb{R}^{n \times n}$ which encodes in its $(i,j)$ entry the ability of node $i$ to communicate with node $j$ using edges {\it up to} time $t$ by counting the walks that have appeared until time $t$. 
For a small time interval $\Delta \ll 1$, 
\begin{equation}\label{eq:constitutive_model}
 S(t + \Delta) = [ I + e^{-b \Delta} S(t) ] [I - \gamma A(t+\Delta) ]^{-\Delta} -I, \quad S(0) = 0, \quad   \gamma,b \in\mathbb{R}_{> 0}.
\end{equation}
For any pair on nodes $i\neq j$ and a single time frame such choice reduces to the classical Katz resolvent-based measure,
\[ S(t_0) + I = (I - \gamma A(t_0))^{-1}, \]
and, more generally, for a discrete-time network sequence $\{t_i\}_{i=1}^{N}$ and $b=0$, to the generalized Katz centrality measure introduced in~\cite{PhysRevE.83.046120},
\[ S(t_N) + I = \prod_{i=1}^{N} (I - \gamma A(t_i) )^{-1}.\]
\begin{remark}\label{rem:temporal_gamma}
From the above equation it follows immediately that, for each resolvent to be well defined, $\gamma$ needs to be smaller than the smallest of all $\rho(A(t_i))^{-1}$. This in turn implies that, in order to compute $S(t_N)+I$, one needs to have complete knowledge of the evolution of the network up to time $t_N$.
\end{remark}
By letting $U(t) = I + S(t)$, expanding in Taylor series to the first order the right-hand side of~\eqref{eq:constitutive_model} and rearranging the terms, one can rewrite the constitutive relation as 
\begin{equation*}
\frac{U(t + \Delta) - U(t)}{\Delta} = b (I-U(t))-U(t) \log (I-\gamma A(t)) + O(\delta ),
\end{equation*}
and thus, by letting $\Delta\to 0$, obtain the non-autonomous Cauchy problem 
\begin{equation}\label{eq:higham_model}
    \begin{cases}
    U'(t) = - b (U(t)-I) - U(t) \log(I - \gamma A(t)), & t > 0,\\
    U(0) = I.
    \end{cases}
\end{equation}
\begin{remark}\label{rem:temporal_gamma2}
Existence of a principal determination of the matrix logarithm function is guaranteed when $ \gamma< \rho(A(t))^{-1}$ for all $t\geq 0$. This implies that, much like in the discrete case, the full temporal evolution of our network has to be known before deriving $U(t)$.
\end{remark}

As suggested in the original paper, alternative approaches can be considered by replacing the resolvent function with an opportune matrix function $f(\gamma A(t))$, i.e., by moving from the Katz centrality measure to a general $f$-centrality,
\begin{equation*}
    \begin{cases}
    W'(t) = - b (W(t)-I) - W(t) \log( f(\gamma A(t)) ), & t > 0,\\
    W(0) = I.
    \end{cases}
\end{equation*}
Just like in the static case we want to employ Mittag--Leffler functions, i.e., $f(\gamma A(t)) = E_\alpha(\gamma A(t))$ for $\alpha\in [0,1]$; this will allow us once again to interpolate between the resolvent and the exponential behavior;
\begin{equation}\label{eq:higham_model_ml}
	    \begin{cases}
	    W'(t) = - b(W(t)-I) - W(t) \log( E_{\alpha}(\gamma A(t)) ), & t > 0,\\
	    W(0) = I,
	    \end{cases}
	\end{equation}
To guarantee the existence of a principal determination of the matrix logarithm function in this case, we simply need $\gamma$ satisfying the requirements in \cref{pro:selecting_b}.
In fact, the striking observation here is that, when $\alpha\in(0,1]$, the choice of $\gamma$ no longer depends on the topology of the temporal network, thus overcoming the issue highlighted in \cref{rem:temporal_gamma,rem:temporal_gamma2}. 

\smallskip
We are now in a position to define centrality measures for nodes in temporal networks. We will define two measures of centrality, one to account for the broadcasting capability of a node, i.e., its ability to spread information to other nodes as time progresses, and one to account for the receiving capability of a node, i.e., its ability to gather information from other nodes and previous time stamps. We notice that, even when the time-evolving network displays only undirected edges, the presence of time induces a sense of  directionality: if information goes from node $i$ to $j$ at time $t$ and then from $j$ to $k$ at time $t+1$, then the information travelled from $i$ to $k$, but not from $k$ to $i$.   
Following~\cite{MR3177230} we thus define the following measures of centrality for temporal networks. 

\begin{definition}
Let $A(t)$ be the adjacency matrix of a time-evolving network $G = (V,E(t))$. Suppose that $\alpha\in (0,1]$ and that $\gamma$ satisfies the conditions of \cref{pro:selecting_b}. Moreover, let $W(t)$ be the solution to \cref{eq:higham_model_ml}. Then, for every node $i\in V$ we define its 
\begin{itemize}
\item \emph{dynamic broadcast centrality} as the $i$th entry of the vector: $\mathbf{b}(t)=W(t)\boldsymbol{1}$; and its
\item \emph{dynamic receive centrality} as the $i$th entry of the vector $\mathbf{r}(t)=W^T(t)\bone$. 	
\end{itemize}
\end{definition}

\begin{remark}
These measures reduce to those introduced in~\cite{MR3177230} when $\alpha = 0$.
\end{remark}

\subsection{Numerical experiments -- continuous time network}

We consider the two synthetic experiments from~\cite{MR3177230} for which we have a way of interpreting the results. The first one models a cascade of information through the directed binary tree structure illustrated in \cref{fig:higham_example_graph}. 
\begin{figure}[t]
    \centering
    \includegraphics[width=0.8\columnwidth]{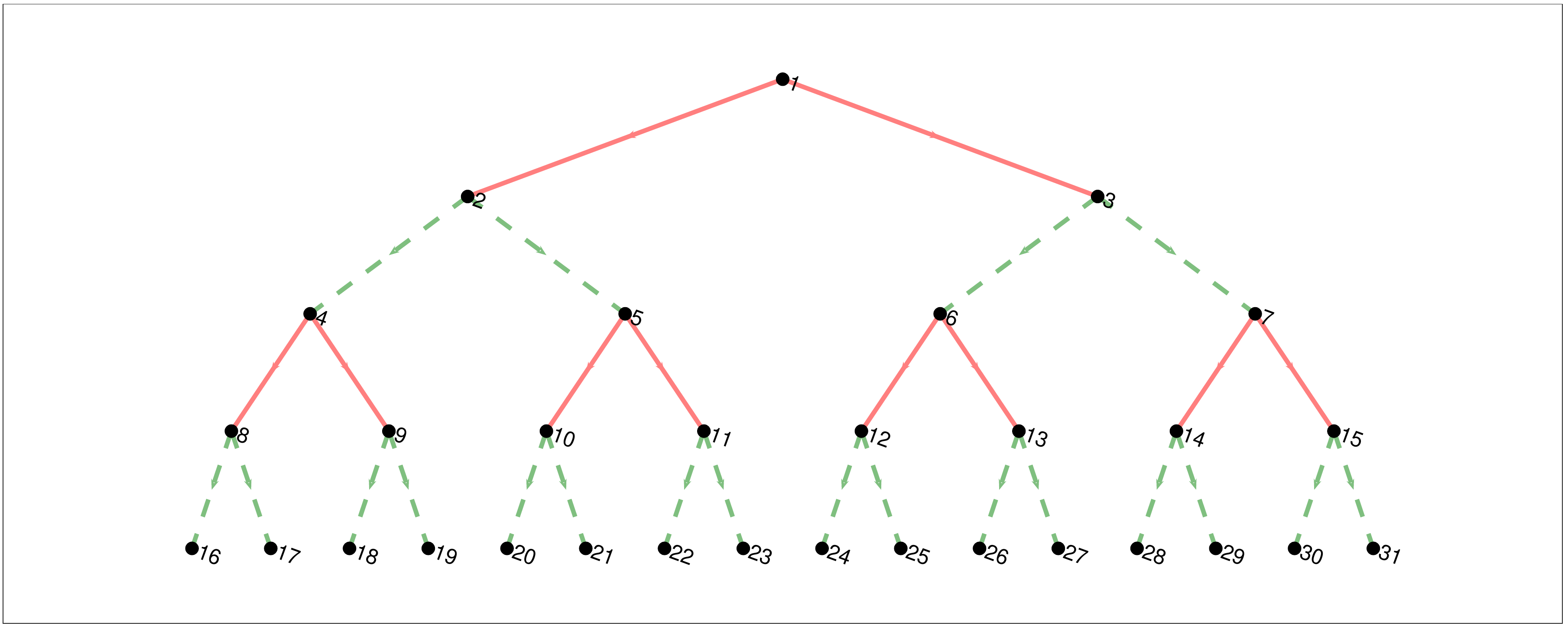}
    \caption{The tree alternates between the solid and dashed edges, i.e., it alternates between two adjacency matrices $A_1$ and $A_2$ made, respectively, by the continuous and dashed edges. In each time step extra noise is added in the form of 5 random directed edges connecting any two nodes of the graph.}
    \label{fig:higham_example_graph}
\end{figure}
On a time interval $T = [0, 20]$, the adjacency matrix $A(t)$ of such network switches between two constant values $A_1$ and $A_2$ on each sub-interval $[i, i + 1)$, specifically
\begin{equation*}
    A(t) = \begin{cases}
    A_1, & \operatorname{mod}(\lfloor t \rfloor,2) = 0 \\
    A_2, & \text{ otherwise},
    \end{cases}
\end{equation*}
where $A_1$ is the adjacency matrix relative to the subgraph with solid edges in \cref{fig:higham_example_graph}, and $A_2$ the one relative to the subgraph with dashed edges. Noise is added to the structure in the form of five extra directed edges chosen uniformly at random at each time interval. The maximum of the spectral radii of all the matrices involved in the computation is $1$, and therefore the solution to \cref{eq:higham_model} is well defined for all $\gamma<1$; see \cref{rem:temporal_gamma2}. As for the time-invariant case, we consider the Kendall $\tau$ correlation between the \emph{broadcast} and \emph{receive} centrality measures obtained by solving~\eqref{eq:higham_model} and the one obtained by solving~\eqref{eq:higham_model_ml}. To compare them we fix for both the same value of $b=0.01$, i.e., a case in which we allow older walks to make a substantial contribution, and compare the measure for the corresponding value of $\gamma$ in \cref{fig:cascade_example}.
\begin{figure}[tbhp]
\centering
\subfloat[Dynamic Receive]{\label{fig:cascade_example_receive}\includegraphics[width=0.45\textwidth]{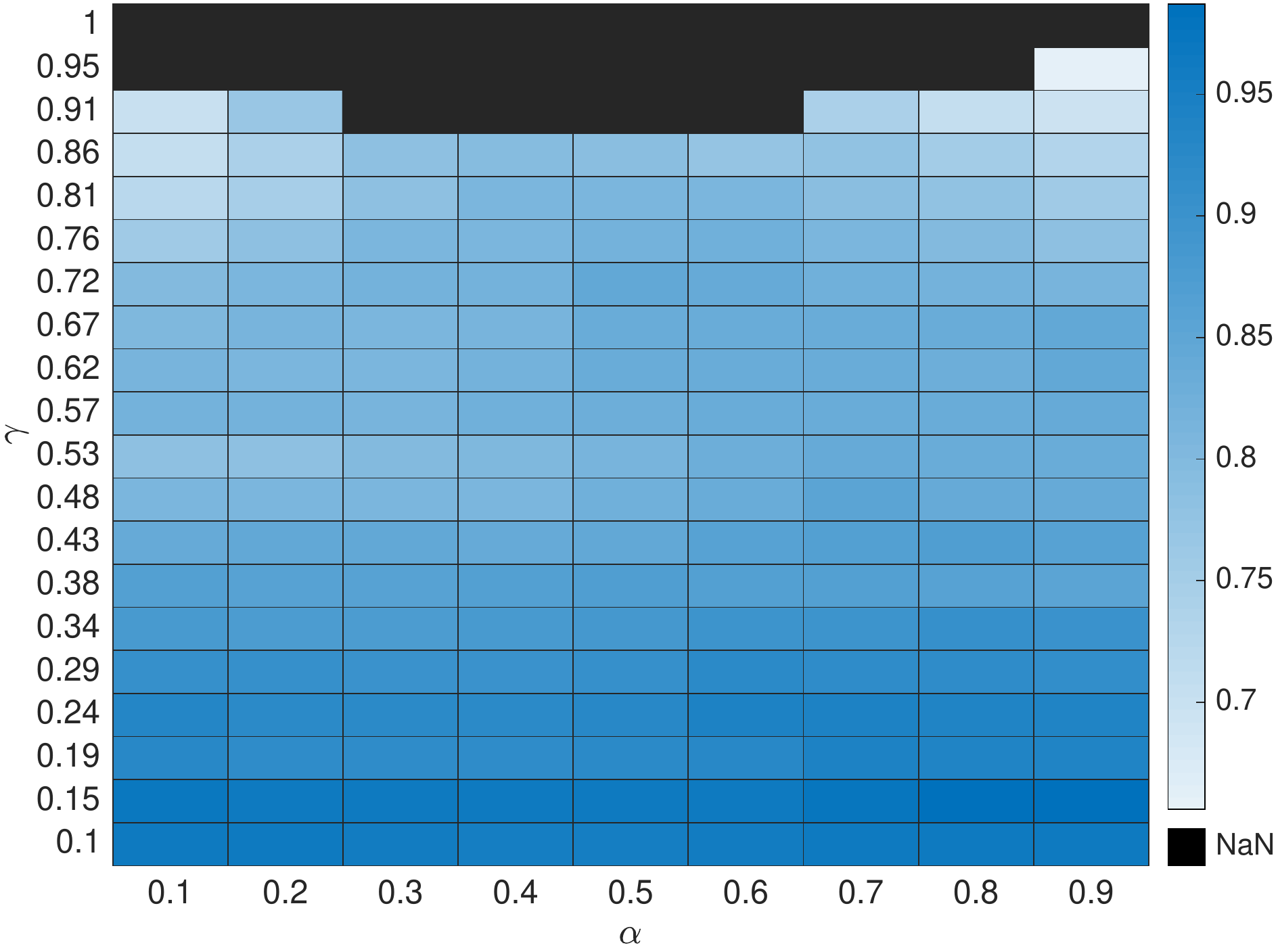}}
\subfloat[Dynamic Broadcast]{\label{fig:cascade_example_broadcast}\includegraphics[width=0.45\textwidth]{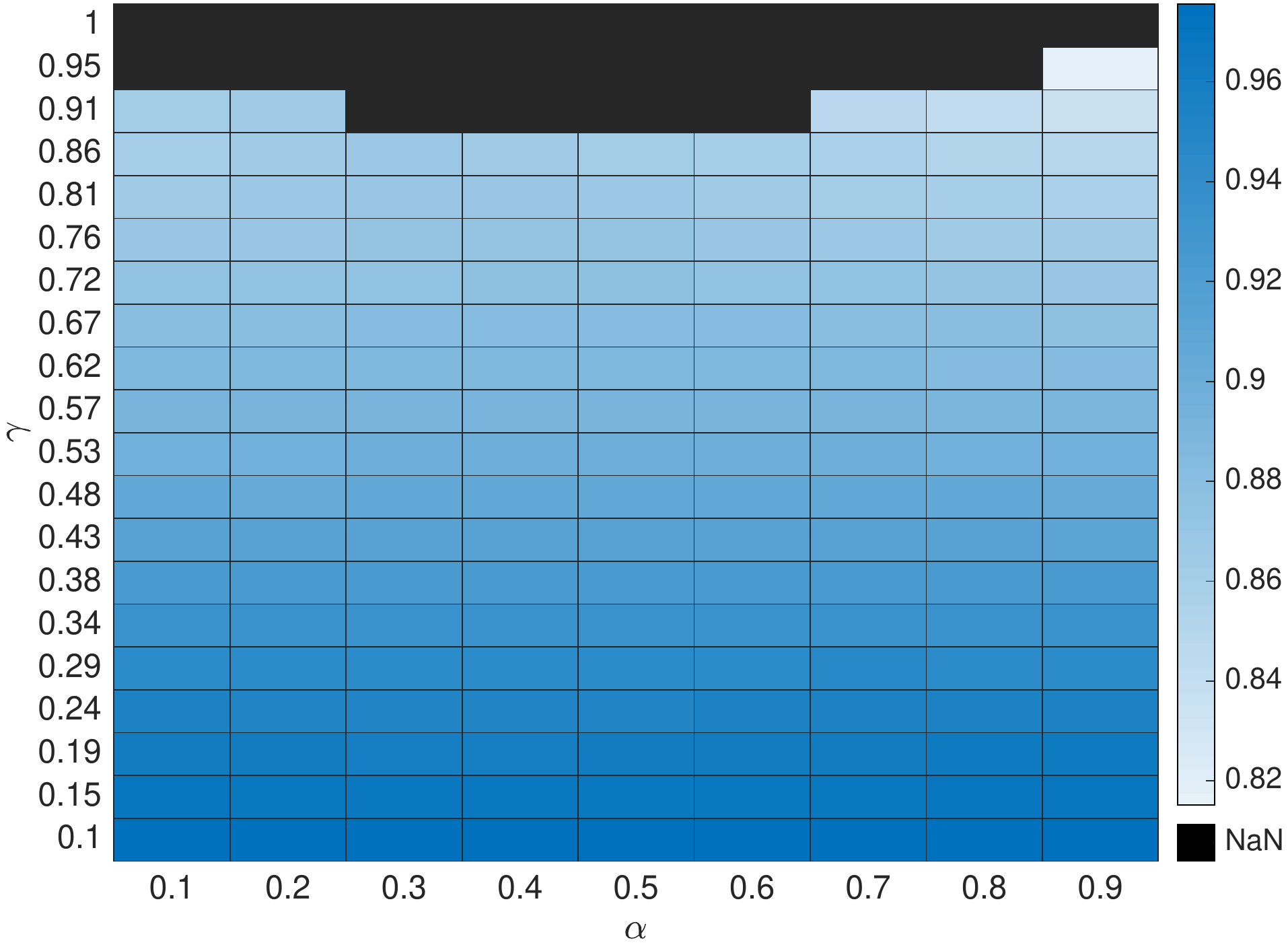}}
\caption{We report here the Kendall-$\tau$ correlation for the receive and broadcast rankings obtained with~\eqref{eq:higham_model} and \eqref{eq:higham_model_ml} with respect to the same $\gamma$ and varying the values of $\alpha$ for the latter.}
\label{fig:cascade_example}
\end{figure}

What we observe for both these measures is that they are more sensitive to the variation of the scaling $\gamma$ than to the variation of $\alpha\to 1$. 

\smallskip

We now consider the second synthetic experiment from~\cite{MR3177230}. This case mimics multiple rounds of voice calls along an undirected tree structure in which every node has at most one edge at any given time, i.e., there are no ``conference'' calls. 
\begin{figure}
    \centering
    \includegraphics[width=\columnwidth]{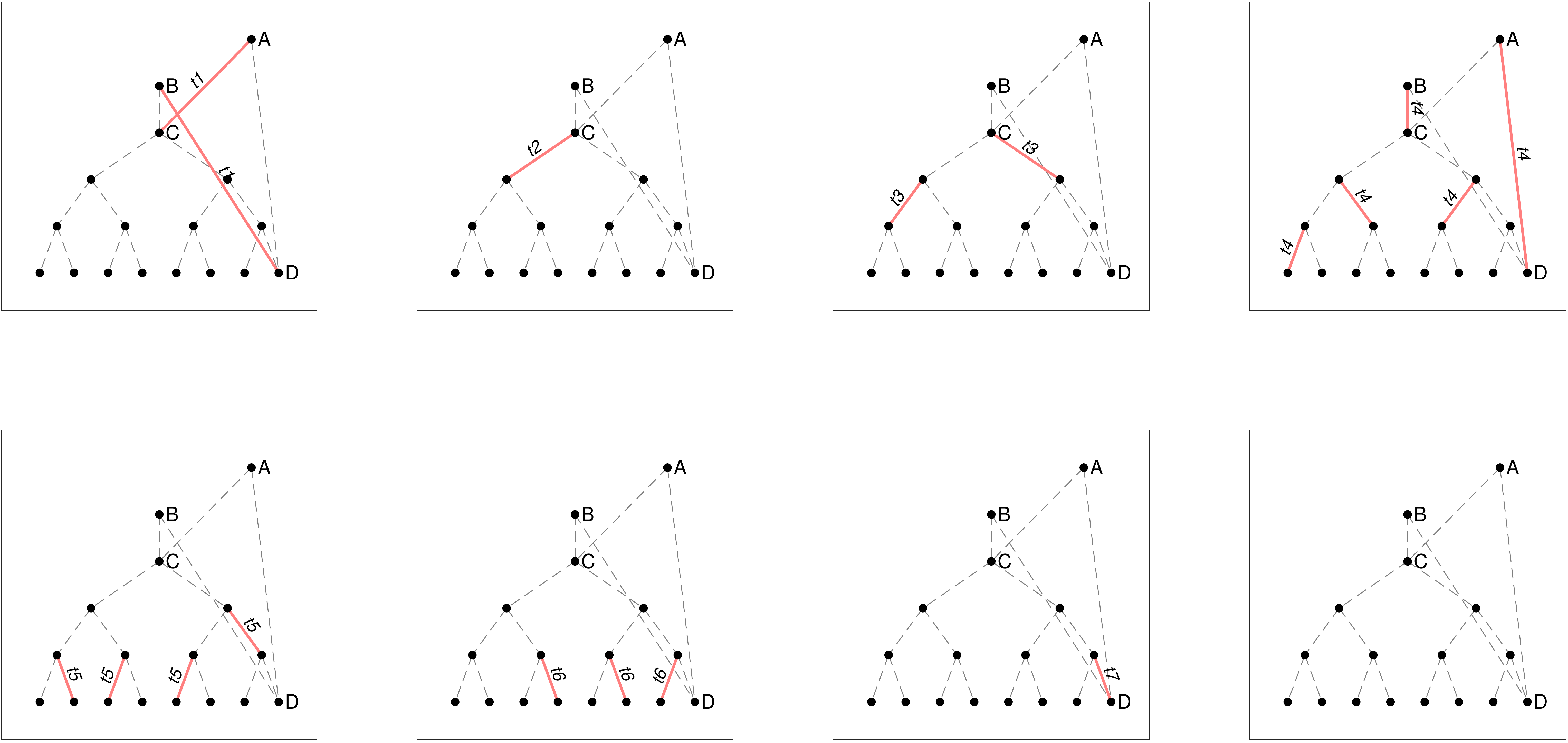}
    \caption{Snapshots of the cycle of activation of the edges in the synthetic phone graph network on the time intervals $t_i = [(i-1)\tau,(i-1+0.9)\tau)$ for $\tau=0.1$, and $i=1,\ldots,8$. The figure reports the arithmetic average over all the time steps for each combination of the parameters (every couple of simulation has been performed to march on the same time-grid).}
    \label{fig:snapshot_telephone}
\end{figure}
\begin{figure}[ht]
    \centering
    \includegraphics[width=0.90\columnwidth]{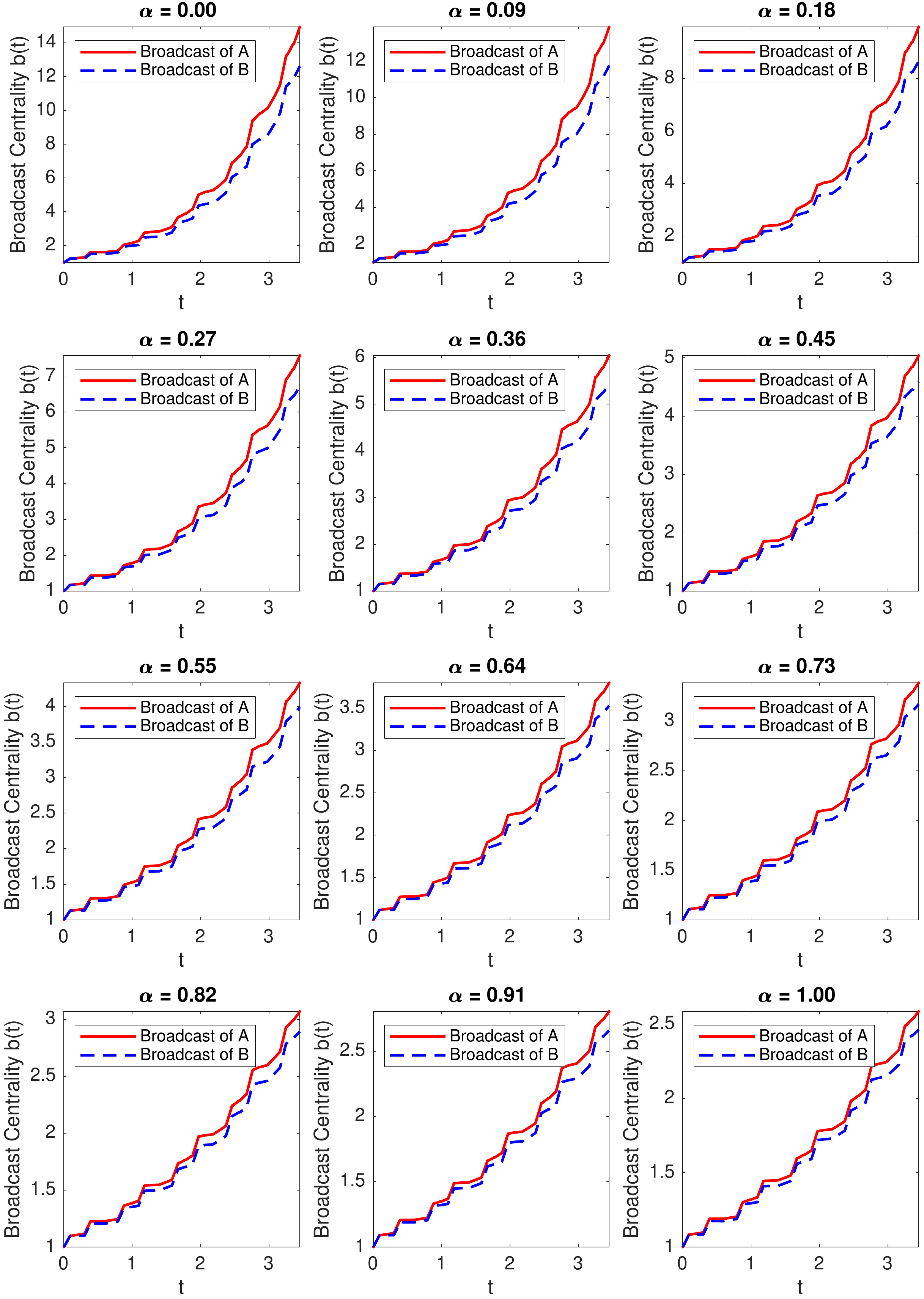}
    \caption{Telephone cascade communication example with model~\eqref{eq:higham_model_ml} for $\gamma=0.9$ and $b=0.1$.}
    \label{fig:telefoncascade}
\end{figure}
In \cref{fig:snapshot_telephone} we have reported the snapshots of the adjacency matrix for the network; all these matrices have unitary spectral radius. 
Connections are built in such a way that node A talks to node C in the first time interval thus initiating the cascade of phone calls in the network. On the other hand, node B waits until the fourth time interval to contact node C, and this does not cause any new cascade of calls.

Even if nodes A and B have an identical behaviour,
both contacting nodes C and D for the same length of time, the results in \cref{fig:telefoncascade} (for $b=0.1$ and $\gamma = 0.9$) show that the dynamic broadcast centrality measure is able to capture the knock-on effect enjoyed by node A irrespective of the value of $\alpha$ used in~\eqref{eq:higham_model_ml}. As we smoothly transition from the resolvent towards the exponential, the same behavior is observed, although with different scales for the centrality scores. 
This confirms that other ML functions allow to replicate the results obtained by resolvent-based temporal measures, while at the same time overcoming the issue of having to select the downweighting parameter $\gamma$; cf.~\cref{rem:temporal_gamma2}.

\section{Conclusions}\label{sec:conclusions} 

We discussed previous appearances of the Mittag--Leffler function $E_{\alpha}(\gamma z)$ in network science and described a general theory for ML-based centrality measures. This new family of functions is parametric, and suitable choices of the parameters were discussed. 
The asymptotics of the centrality measures  were discussed theoretically and numerically, showing that by varying $(\alpha,\gamma) \in [0,1]\times (0,\infty)$ our centrality indices move between degree, eigenvector, resolvent, and exponential centrality indices. 
We described new ML-based centrality measures for time-evolving networks by extending previous results based on the matrix resolvent. We introduced two parametric centrality measures for which the parameter no longer depends on the underlying dynamic graph, thus allowing for greater flexibility in the implementation of these techniques. 

Numerical experiments on both real-world and synthetic networks were presented.

Future work will focus on exploiting the connection linking Mittag--Leffler functions and the evolution of dynamical systems with respect to a time-fractional derivative. In particular, we plan to analyze the behavior of networked dynamical systems evolving in fractional-time by means of ML functions.

\section*{Acknowledgments}
The authors thank the anonymous referees for their valuable suggestions. 

\bibliographystyle{siamplain}
\bibliography{fractionalnetwork}



\section*{Supplementary material}

# SUPPLEMENTARY MATERIALS: MITTAG–LEFFLER FUNCTIONS AND THEIR APPLICATIONS IN NETWORK SCIENCE[*]

FRANCESCA ARRIGO[†] AND FABIO DURASTANTE[‡]

**SM1. Additional numerical results.** Here we summarize some additional results on the correlation between the rankings obtained with different centrality measures for the networks NEWMAN/DOLPHINS and GLEICH/MINNESOTA, both freely available at [SM2]. In particular, in Tables SM1 and SM2 the top 20% of nodes ranked according to degree centrality, eigenvector centrality and either ML-subgraph centrality or ML-total communicability for the network NEWMAN/DOLPHINS are listed. It is immediate to see that, whenever well defined (that is, for values of $\alpha$ and $\gamma$ below the red line in Figures 4,5), the ML-based centrality measures return results comparable with those presented in the manuscript for the whole network.

The results are not as good for GLEICH/MINNESOTA when eigenvector centrality is considered, as it can be evinced from Tables SM3 and SM4, where the top ten nodes according to different measures are listed. However, this does not come as a surprise. Indeed, poor correlation with eigenvector centrality follows from the fact that the network displays a small spectral gap, i.e., the difference between the largest eigenvalue and the second largest eigenvalue is very small $\lambda_1 - \lambda_2 = 4.53e - 04$; see [SM1]. We also want to remark that the first column in these two tables differ from each other because all nodes listed here (with the only exception of node 2418) have the same degree, equal to four, and are thus ranked joint second.

---

[†]Department of Mathematics and Statistics, University of Strathclyde, Glasgow, UK (francesca.arrigo@strath.ac.uk,).

[‡]Dipartimento di Matematica, Università di Pisa, Pisa, IT (fabio.durastante@di.unipi.it), Istituto per le Applicazioni del Calcolo "M. Picone", Consiglio Nazionale delle Ricerche, Napoli, IT (f.durastante@na.iac.cnr.it).

TABLE SM1

NEWMAN/DOLPHINS: *Labels of the top 20% of ranked nodes according to degree centrality (degree), ML-subgraph centrality for different values of $\alpha$ and $\gamma$ $(E_\alpha(\gamma A)_{ii})$, and eigenvector centrality (evec).*

| $d_i$ | $s_i(\widetilde{E_\alpha})$ | | | | | | | | | | | | | | | | | | | | | $x_i$ |
|---|---|---|---|---|---|---|---|---|---|---|---|---|---|---|---|---|---|---|---|---|---|---|
| | $\alpha$ | 0.1 | | | | | 0.35 | | | | | 0.6 | | | | | 0.85 | | | | | |
| | $\gamma$ | 0.1 | 0.325 | 0.55 | 0.755 | 1 | 0.1 | 0.325 | 0.55 | 0.755 | 1 | 0.1 | 0.325 | 0.55 | 0.755 | 1 | 0.1 | 0.325 | 0.55 | 0.755 | 1 | |
| 15 | | 15 | — | — | — | — | 15 | 15 | 15 | 15 | 15 | 15 | 15 | 15 | 15 | 15 | 15 | 15 | 15 | 15 | 15 | 15 |
| 38 | | 46 | — | — | — | — | 46 | 38 | 38 | 38 | 38 | 38 | 38 | 38 | 38 | 38 | 46 | 46 | 38 | 38 | 38 | 38 |
| 46 | | 38 | — | — | — | — | 38 | 46 | 46 | 46 | 46 | 46 | 46 | 46 | 46 | 46 | 38 | 38 | 46 | 46 | 46 | 46 |
| 34 | | 34 | — | — | — | — | 34 | 34 | 34 | 34 | 34 | 34 | 34 | 34 | 34 | 34 | 34 | 34 | 34 | 34 | 34 | 34 |
| 52 | | 52 | — | — | — | — | 52 | 51 | 51 | 51 | 51 | 52 | 52 | 51 | 51 | 51 | 52 | 52 | 52 | 51 | 51 | 51 |
| 18 | | 58 | — | — | — | — | 58 | 30 | 30 | 30 | 30 | 58 | 30 | 30 | 30 | 30 | 58 | 30 | 30 | 52 | 52 | 30 |
| 21 | | 30 | — | — | — | — | 30 | 52 | 52 | 52 | 52 | 30 | 51 | 52 | 52 | 52 | 30 | 19 | 51 | 30 | 30 | 52 |
| 30 | | 14 | — | — | — | — | 14 | 17 | 17 | 17 | 17 | 19 | 30 | 17 | 17 | 17 | 14 | 58 | 19 | 17 | 17 | 17 |
| 58 | | 18 | — | — | — | — | 18 | 41 | 41 | 41 | 41 | 51 | 17 | 41 | 41 | 41 | 21 | 14 | 14 | 41 | 41 | 41 |
| 2 | | 19 | — | — | — | — | 21 | 22 | 22 | 22 | 22 | 17 | 22 | 22 | 22 | 22 | 18 | 18 | 17 | 22 | 22 | 22 |
| 14 | | 39 | — | — | — | — | 39 | 19 | 19 | 19 | 19 | 22 | 14 | 19 | 19 | 19 | 30 | 41 | 41 | 19 | 19 | 19 |
| 39 | | 41 | — | — | — | — | 41 | 39 | 39 | 39 | 39 | 41 | 39 | 39 | 39 | 39 | 41 | 39 | 39 | 39 | 39 | 39 |

TABLE SM2

NEWMAN/DOLPHINS: Labels of the top 20% of ranked nodes according to degree centrality (degree), ML-total communicability for different values of $\alpha$ and $\gamma$ $(E_\alpha(\gamma A)_{ii})$, and eigenvector centrality (evec).

| $d_i$ | $t_i(\widetilde{E}_\alpha)$ | | | | | | | | | | | | | | | | | | | | | | | | | $x_i$ |
|---|---|---|---|---|---|---|---|---|---|---|---|---|---|---|---|---|---|---|---|---|---|---|---|---|---|---|
| | $\alpha$ | 0.1 | | | | 0.35 | | | | 0.6 | | | | 0.85 | | | | |
| | $\gamma$ | 0.1 | 0.325 | 0.55 | 0.755 | 1 | 0.1 | 0.325 | 0.55 | 0.755 | 1 | 0.1 | 0.325 | 0.55 | 0.755 | 1 | 0.1 | 0.325 | 0.55 | 0.755 | 1 | |
| 15 | | 15 | — | — | — | — | 15 | 15 | 15 | 15 | 15 | 15 | 15 | 15 | 15 | 15 | 15 | 15 | 15 | 15 | 15 | 15 |
| 38 | | 38 | — | — | — | — | 38 | 38 | 38 | 38 | 38 | 38 | 38 | 38 | 38 | 38 | 38 | 38 | 38 | 38 | 38 | 38 |
| 46 | | 46 | — | — | — | — | 46 | 46 | 46 | 46 | 46 | 46 | 46 | 46 | 46 | 46 | 46 | 46 | 46 | 46 | 46 | 46 |
| 34 | | 34 | — | — | — | — | 34 | 34 | 34 | 34 | 34 | 34 | 34 | 34 | 34 | 34 | 34 | 34 | 34 | 34 | 34 | 34 |
| 52 | | 52 | — | — | — | — | 51 | 51 | 51 | 51 | 51 | 51 | 51 | 51 | 51 | 52 | 41 | 51 | 51 | 51 | 51 | 51 |
| 18 | | 30 | — | — | — | — | 30 | 30 | 30 | 30 | 30 | 30 | 30 | 30 | 30 | 30 | 51 | 41 | 30 | 30 | 30 | 30 |
| 21 | | 41 | — | — | — | — | 52 | 52 | 52 | 52 | 52 | 52 | 52 | 52 | 52 | 21 | 52 | 30 | 52 | 41 | 52 | 52 |
| 30 | | 51 | — | — | — | — | 41 | 17 | 17 | 17 | 17 | 41 | 17 | 41 | 17 | 41 | 30 | 52 | 17 | 52 | 17 | 17 |
| 58 | | 21 | — | — | — | — | 17 | 41 | 41 | 41 | 41 | 17 | 52 | 17 | 41 | 58 | 21 | 17 | 41 | 17 | 41 | 41 |
| 2 | | 39 | — | — | — | — | 22 | 22 | 22 | 22 | 39 | 22 | 22 | 22 | 22 | 18 | 19 | 22 | 22 | 22 | 22 | 22 |

TABLE SM3

GLEICH/MINNESOTA: *Labels of the top 10 ranked nodes according to degree centrality (degree), ML-subgraph centrality for different values of $\alpha$ and $\gamma$ $(E_\alpha(\gamma A)_{ii})$, and eigenvector centrality (evec).*

| $d_i$ | $s_i(\widetilde{E_\alpha})$ | | | | | | | | | | | | | | | | | | | | | | $x_i$ |
|---|---|---|---|---|---|---|---|---|---|---|---|---|---|---|---|---|---|---|---|---|---|---|---|
| | $\alpha$ | 0.1 | | | | | 0.35 | | | | | | 0.6 | | | | | | 0.85 | | | | |
| | $\gamma$ | 0.1 | 0.325 | 0.55 | 0.755 | 1 | 0.1 | 0.325 | 0.55 | 0.755 | 1 | 0.1 | 0.325 | 0.55 | 0.755 | 1 | 0.1 | 0.325 | 0.55 | 0.755 | 1 | |
| 2418 | | 2418 | 1927 | 1927 | — | — | 2418 | 891 | 891 | 1927 | 1927 | 2418 | 891 | 891 | 891 | 891 | 2418 | 891 | 891 | 891 | 891 | 1927 |
| 891  | | 891  | 891  | 1930 | — | — | 891  | 815 | 1927 | 1930 | 1930 | 891  | 815 | 1927 | 1927 | 1927 | 891  | 815 | 815 | 815 | 815 | 1930 |
| 815  | | 815  | 1930 | 1987 | — | — | 815  | 806 | 1778 | 1987 | 1987 | 815  | 806 | 1778 | 1778 | 1778 | 815  | 806 | 806 | 806 | 806 | 1987 |
| 806  | | 806  | 1778 | 1933 | — | — | 806  | 2489 | 1987 | 1912 | 1933 | 806  | 2489 | 2418 | 806  | 1987 | 806  | 2489 | 2489 | 2489 | 2489 | 1933 |
| 2489 | | 2489 | 1987 | 1912 | — | — | 2489 | 1927 | 1930 | 1933 | 1912 | 2489 | 1927 | 2489 | 815  | 1788 | 2489 | 2418 | 2418 | 1927 | 1927 | 1932 |
| 1927 | | 1927 | 1788 | 1932 | — | — | 1927 | 1778 | 1788 | 1912 | 1933 | 1927 | 1778 | 1927 | 1987 | 1927 | 1927 | 1927 | 1927 | 1987 | 1778 | 1988 |
| 1987 | | 1987 | 1912 | 1988 | — | — | 1987 | 1987 | 1919 | 1948 | 1919 | 1987 | 1987 | 1778 | 905  | 1930 | 1987 | 1987 | 1987 | 1778 | 1987 | 1885 |
| 1778 | | 1778 | 1919 | 1919 | — | — | 1778 | 1778 | 1932 | 1948 | 1932 | 1778 | 905  | 1930 | 895  | 905  | 1778 | 1778 | 1778 | 905  | 905  | 1912 |
| 1009 | | 1009 | 1933 | 1948 | — | — | 1009 | 1778 | 905  | 1778 | 1778 | 1009 | 905  | 895  | 1788 | 895  | 1009 | 1009 | 905  | 895  | 1880 | 1880 |
| 1225 | | 1225 | 1948 | 1885 | — | — | 1225 | 2272 | 1948 | 1988 | 1917 | 1225 | 2272 | 2272 | 2272 | 815  | 1225 | 1225 | 1009 | 2272 | 2272 | 1919 |

TABLE SM4

GLEICH/MINNESOTA: Labels of the top 10 ranked nodes according to degree centrality (degree), ML-total communicability for different values of $\alpha$ and $\gamma$ ($E_\alpha(\gamma A)\mathbf{1}_i$), and eigenvector centrality (evec).

| $d_i$ | $t_i(\widetilde{E}_\alpha)$ | | | | | | | | | | | | | | | | | | | | | $x_i$ |
|---|---|---|---|---|---|---|---|---|---|---|---|---|---|---|---|---|---|---|---|---|---|---|
| | $\alpha$ | 0.1 | | | | 0.35 | | | | 0.6 | | | | 0.85 | | | | | | | | |
| | $\gamma$ | 0.1 | 0.325 | 0.55 | 0.755 | 1 | 0.1 | 0.325 | 0.55 | 0.755 | 1 | 0.1 | 0.325 | 0.55 | 0.755 | 1 | 0.1 | 0.325 | 0.55 | 0.755 | 1 | |
| 2418 | | 2418 | 1912 | 1927 | – | – | 2418 | 1788 | 1912 | 1912 | 1912 | 2418 | 1788 | 1788 | 1912 | 1912 | 2418 | 1788 | 1788 | 1788 | 1788 | 1927 |
| 1788 | | 1788 | 1919 | 1912 | – | – | 1788 | 1927 | 1919 | 1919 | 1919 | 1788 | 1927 | 1927 | 1927 | 1919 | 1788 | 1927 | 1927 | 1927 | 1927 | 1930 |
| 1927 | | 1927 | 1927 | 1919 | – | – | 1927 | 1919 | 1927 | 1948 | 1948 | 1927 | 1778 | 1919 | 1919 | 1927 | 1987 | 1987 | 1778 | 1919 | 1919 | 1987 |
| 1987 | | 1987 | 1948 | 1930 | – | – | 1987 | 1912 | 1948 | 1927 | 1927 | 1987 | 1919 | 1912 | 1948 | 1948 | 1778 | 1778 | 1919 | 1778 | 1912 | 1933 |
| 1778 | | 1778 | 1930 | 1948 | – | – | 1778 | 1778 | 1930 | 1930 | 1930 | 1778 | 1987 | 1778 | 1788 | 1930 | 1919 | 1919 | 1987 | 1912 | 1778 | 1932 |
| 1919 | | 1919 | 1884 | 1987 | – | – | 1919 | 1948 | 1919 | 1884 | 1950 | 1919 | 1948 | 1948 | 1778 | 1788 | 1948 | 1948 | 1948 | 1948 | 1948 | 1988 |
| 1948 | | 1948 | 1933 | 1933 | – | – | 1948 | 1987 | 1948 | 1950 | 1917 | 1948 | 1917 | 1987 | 1948 | 1778 | 1917 | 1917 | 1912 | 1987 | 1987 | 1885 |
| 1917 | | 1917 | 1987 | 1932 | – | – | 1917 | 1930 | 1987 | 1987 | 1884 | 1917 | 1912 | 1930 | 1987 | 1933 | 1948 | 1933 | 1917 | 1917 | 1930 | 1912 |
| 1933 | | 1933 | 1788 | 1988 | – | – | 1933 | 1917 | 1933 | 1917 | 1946 | 1946 | 1933 | 1917 | 1933 | 1884 | 1917 | 1946 | 1933 | 1930 | 1917 | 1880 |
| 1946 | | 1946 | 1778 | 1950 | – | – | 1946 | 1933 | 1987 | 1933 | 1909 | 1933 | 1946 | 1933 | 1917 | 1987 | 1933 | 2272 | 1946 | 1933 | 1933 | 1919 |



\end{document}